\documentclass[leqno,12pt]{article}
\usepackage{amsmath}
\usepackage[latin1]{inputenc}
\usepackage{epsfig}
\usepackage{amssymb}
\usepackage{supertabular}
\usepackage{graphicx}
\usepackage{color}

\input{tcilatex}
\begin{document}

\title{ An $hp$ Finite Element Method for singularly perturbed transmission
problems in smooth domains}
\author{Serge Nicaise \\
Universit\'e de Valenciennes et du Hainaut Cambr\'esis\\
LAMAV, FR CNRS 2956, \\
Institut des Sciences et Techniques de Valenciennes\\
F-59313 - Valenciennes Cedex 9 France\\
Serge.Nicaise@univ-valenciennes.fr\\
\\
Christos Xenophontos\\
Department of Mathematics and Statistics\\
University of Cyprus\\
P.O. BOX 20537\\
Nicosia 1678 Cyprus\\
xenophontos@ucy.ac.cy\\
}
\maketitle

\begin{abstract}
We consider a two-dimensional singularly perturbed transmission problem with
two different diffusion coefficients, in a domain with smooth (analytic)
boundary. The solution will contain boundary layers only in the part of the
domain where the diffusion coefficient is high and interface layers along
the interface. Utilizing existing and newly derived regularity results for
the exact solution, we design a robust $hp$ finite element method for its
approximation. Under the assumption of analytic input data, we show that the
method converges at an e{xponential} rate, provided the mesh and polynomial
degree distribution are chosen appropriately. Numerical results illustrating
our theoretical findings are also included.
\end{abstract}



\section{Introduction}

The approximation of singularly perturbed problems has retained the
attention of many authors in recent years. Let us mention \cite%
{melenk,miller:96,morton:96,roos:96,schwab,ss} and the references quoted
there. However, in all references quoted no analysis is carried out for
differential operators with piecewise constant or piecewise smooth
coefficients. On the other hand, in many real life applications, the
differential operators have such piecewise coefficients that may have a very
large discrepancy. In that case, the solution of the problem will contain
boundary layers near the exterior boundary (as usual) but will also contain
interface layers along the interface where the coefficients have a large
jump. We refer to \cite{mn} for the description of this phenomenon in one
and two dimensions and to \cite{nx} for several numerical methods for the
robust approximation of such problems in one-dimension.

The goal of the present paper is to extend certain results from \cite{nx} to
two-dimensions. \ In particular, we consider a singularly perturbed
transmission problem in a domain with analytic boundary. \ Under the
assumption of the data also being analytic, we provide an asymptotic
expansion for the solution (in the style of \cite{melenkB}) that provides
the necessary information for the design of a robust finite element method
that converges at an exponential rate as the degree $p$ of the approximating
polynomials is increased. The expansion of the solution includes an outer
(smooth) part, an inner (boundary layer) part, an interface layer and a
(smooth) remainder. \ \ The regularity of each compoment is studied and
known results from \cite{ms1} allow us to treat the outer and inner parts,
as well as the remainder (defined on one part of the domain). The results
obtained for the regularity of the interface layer (and the remainder
defined on the other part of the domain) are new and in line with those
reported in \cite{nx} for the one-dimensional analog of our model problem.
Our work closely follows what was done in \cite{ms} but also includes the
additional analysis for the interface layer.

The paper is organized as follows: In Section \ref{sect1} we present the
singularly perturbed problem and describe the typical phenomena. Section \ref%
{expansion} is devoted to the expansion of the solution of our model problem
into the parts mentioned above (i.e. outer, inner, interface and remainder).
The regularity of each component is also described in that section. Section {%
\ref{approx}} gives the main approximation result and in Section \ref%
{computations} we show the results of numerical computations illustrating
our theoretical findings. We end with some conclusions in Section \ref%
{conclusions}.

Throughout the paper the spaces $H^{s}(\Omega )$, with $s\geq 0$, are the
standard Sobolev spaces on the domain $\Omega \subset \mathbb{R}^{2}$, with
norm $\Vert \cdot \Vert _{s,\Omega }$ and semi-norm $|\cdot |_{s,\Omega }$.
The space $H_{0}^{1}(\Omega )$ is defined, as usual, by $H_{0}^{1}(\Omega
):=\{v\in H^{1}(\Omega ):\left. v\right\vert _{\partial \Omega }=0\}$. $%
L^{p}(\Omega )$, $p>1$, are the usual Lebesgue spaces with norm $\Vert \cdot
\Vert _{0,p,\Omega }$ (we drop the index $p$ for $p=2$). Finally, the
notation $A\lesssim B$ means the existence of a positive constant $C$, which
is independent of the quantities $A$ and $B$ under consideration and of the
parameter ${\varepsilon }$, such that $A\leq CB$.


\section{The model problem\label{sect1}}

\noindent Let $\Omega _{+}$ and $\Omega _{-}$ be smooth domains in $\mathbb{R%
}^{2}$, with respective boundaries $\partial \Omega _{+}$ and $\partial
\Omega _{-},$ such that $\partial \Omega _{+}\cap \partial \Omega
_{-}=\Sigma $; an example is shown in Figure 1 below. We assume that $%
\partial \Omega$ is an analytic curve, i.e. $\partial \Omega_{\pm}$ and $%
\Sigma$ are analytic curves. Moreover, we assume that $\partial \Omega_+
\backslash \Sigma$, as well as $\Sigma$ are connected. We will write $%
\Omega=\Omega _{+}\cup \Omega _{-}$, and for any function $u$ defined on $%
\Omega $ we will denote by $u_{+} $ (resp. $u_{-}$) the restriction of $u$
to $\Omega_{+}$ (resp. $\Omega _{-}$) and we will write $u\equiv \left(
u_{+},u_{-}\right)$.

\begin{figure}[h]
\setlength{\unitlength}{1cm}
\par
\begin{center}
\hspace{2cm} \psfig{figure=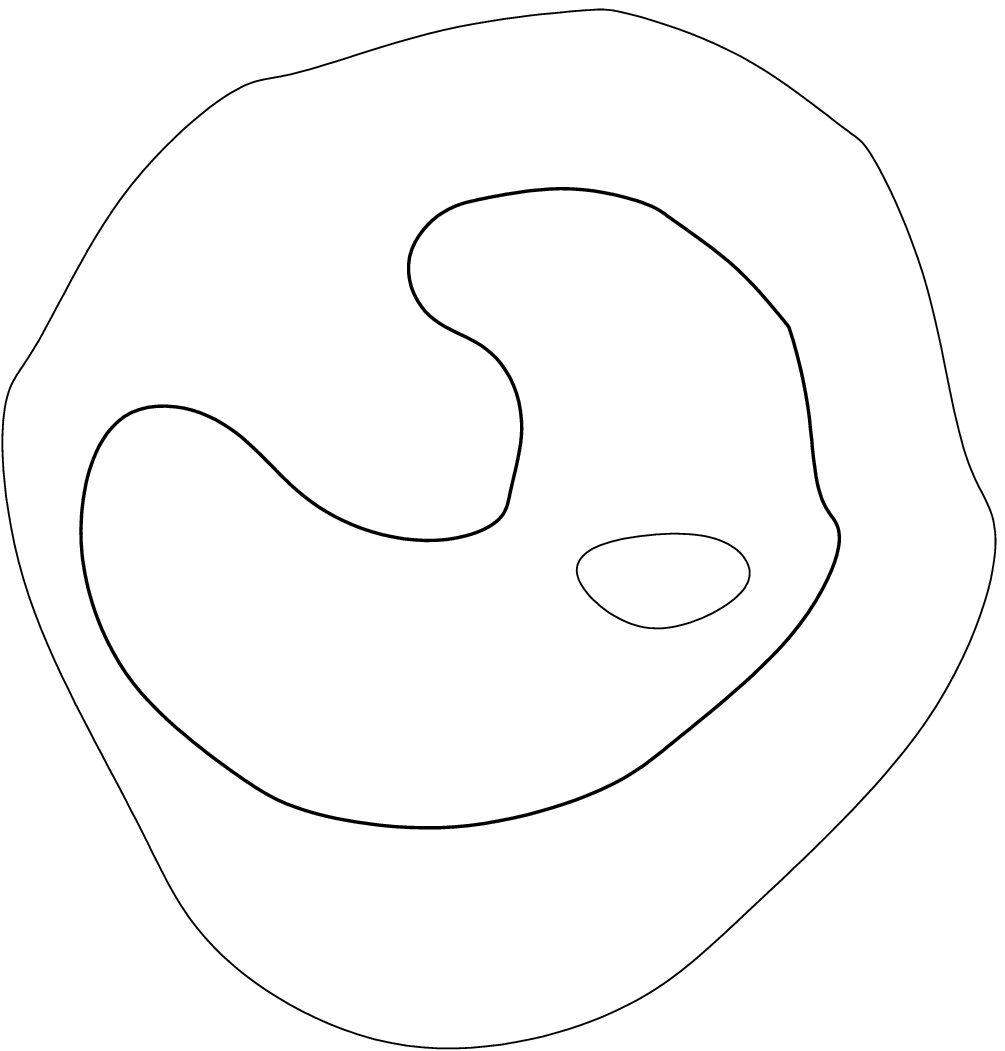,width=4cm} 
\begin{picture}(2,3)
\put (-2.5,1.25) {\shortstack[c] {$\Omega_{+}$}}
\put (-3.2,2.95) {\shortstack[c] {$\Omega_{-}$}}
\put (-1.2,3.25) {\shortstack[c] {$\Sigma$}}
\end{picture}
\end{center}
\par
\vspace{-0.5cm}
\caption{Example of the domains $\Omega_+$ and $\Omega_-$.}
\end{figure}

We consider the following singularly perturbed transmission problem: Find $%
u^{\varepsilon }=\left( u_{+}^{\varepsilon },u_{-}^{\varepsilon }\right) $
such that 
\begin{eqnarray}  \label{bvp}
-\varepsilon ^{2}\Delta u_{+}^{\varepsilon }+u_{+}^{\varepsilon } &=&f_{+}%
\text{ in }\Omega _{+},  \label{bvpa} \\
-\Delta u_{-}^{\varepsilon }+u_{-}^{\varepsilon } &=&f_{-}\text{ in }\Omega
_{-},  \label{bvpb} \\
u_{+}^{\varepsilon } &=&0\text{ on }\partial \Omega _{+}\backslash \Sigma ,
\label{bvpc} \\
u_{-}^{\varepsilon } &=&0\text{ on }\partial \Omega _{-}\backslash \Sigma ,
\label{bvpd} \\
u_{+}^{\varepsilon }-u_{-}^{\varepsilon } &=&0\text{ on }\Sigma ,
\label{bvpe} \\
\varepsilon ^{2}\frac{\partial u_{+}^{\varepsilon }}{\partial \nu }-\frac{%
\partial u_{-}^{\varepsilon }}{\partial \nu } &=&h\text{ on }\Sigma ,
\label{bvpf}
\end{eqnarray}
where $\Delta $ denotes the Laplacian operator, $\varepsilon \in (0,1]$ is a
given parameter, $f_{\pm },h$ are given smooth functions and $\nu $ denotes
the outward normal vector along $\Sigma $ oriented outside $\Omega_+$. The
formal limit problem of (\ref{bvpa})--(\ref{bvpf}), as $\varepsilon
\rightarrow 0$, is 
\begin{eqnarray*}
u_{+}^{0} &=&f_{+}\text{ in }\Omega _{+}, \\
-\Delta u_{-}^{0}+u_{-}^{0} &=&f_{-}\text{ in }\Omega _{-}, \\
u_{+}^{0} &=&0\text{ on }\partial \Omega _{+}\backslash \Sigma , \\
u_{-}^{0} &=&0\text{ on }\partial \Omega _{-}\backslash \Sigma , \\
u_{+}^{0}-u_{-}^{0} &=&0\text{ on }\Sigma , \\
-\frac{\partial u_{-}^{0}}{\partial \nu } &=&h\text{ on }\Sigma .
\end{eqnarray*}
Since, in general, $f_{+}$ does not satisfy the boundary conditions $%
f_{+}=u_{+}^{0}$ on $\partial \Omega _{+}\backslash \Sigma $ and $%
f_{+}=u_{-}^{0}$ on $\Sigma $, we expect that the solution $u^{\varepsilon }$
will contain boundary layers along $\partial \Omega _{+}\backslash \Sigma $
and an interface layer along $\Sigma $.

We assume that the data of our problem is analytic and satisfies 
\begin{equation}
\left\Vert \nabla ^{p}f_{\pm }\right\Vert _{\infty ,\Omega _{\pm }}\leq
C_{f_{\pm }}\gamma _{f_{\pm }}^{p}p!\ \forall \ p=0,1,2,...,  \label{assf}
\end{equation}
\begin{equation}
\left\Vert \nabla_{\Sigma}^{p}h\right\Vert _{\infty ,\Sigma }\leq
C_{h}\gamma _{h}^{p}p!\ \forall \ p=0,1,2,...,  \label{assh}
\end{equation}
for some positive constants $C_{f_{\pm }},\gamma _{f_{\pm }},C_{h},\gamma
_{h}$, where $\nabla_{\Sigma}$ denotes the tangential derivative along $%
\Sigma$. The following theorem gives bounds on the derivatives of the
solution to (\ref{bvpa})--(\ref{bvpf}) that are explicit in terms of the
order of differentiation as well as the singular perturbation parameter $%
\varepsilon$.

\begin{theorem}
\label{thm_asy}
Let $u^{\varepsilon }=\left( u_{+}^{\varepsilon },u_{-}^{\varepsilon
}\right) $ be the solution to (\ref{bvpa})--(\ref{bvpf}) with the data
satysfying (\ref{assf}), (\ref{assh}). \ Then there are constants $C,K>0$ 
depending only of the data such that
\begin{equation}
\varepsilon\left\Vert D^{\alpha }u_{+}^{\varepsilon }\right\Vert _{0,\Omega _{+}}+
\left\Vert D^{\alpha }u_{-}^{\varepsilon }\right\Vert _{0,\Omega _{-}}\leq
C \varepsilon K^{\left\vert \alpha \right\vert }\max \left\{ \left\vert \alpha
\right\vert ,\varepsilon ^{-1}\right\} ^{\left\vert \alpha \right\vert
}\;\forall \;\alpha = 1,2,... .  \label{Duserge}
\end{equation}
%
%
\end{theorem}%
\begin{proof} 
This follows from the local estimates
\begin{eqnarray*}
\varepsilon |u_{+}^{\varepsilon }| _{2, B_{x_0}\cap \Omega _{+}}+
|u_{-}^{\varepsilon }|_{2, B_{x_0}\cap \Omega _{-}}\leq
C (\varepsilon^{-1} \|f_+\|_{0, B'_{x_0}\cap \Omega _{+}}\\
+\|f_-\|_{0, B'_{x_0}\cap \Omega _{-}}+\|h\|_{B'_{x_0}\cap\Sigma}
+\varepsilon \|u_{+}^{\varepsilon }\| _{1, B'_{x_0}\cap \Omega _{+}}+
\|u_{-}^{\varepsilon }\|_{1, B'_{x_0}\cap \Omega _{-}}),
\end{eqnarray*}
for all sufficiently small balls $\bar B_{x_0}\subset B'_{x_0}$ centred at $x_0\in \Sigma$ (proved by a local change of variables and some reflexions to reduce the transmission problem into a Dirichlet problem
and a Neumann one in  half-balls)
and the use of Morrey-Nirenberg techniques (see Theorem 2.1 in \cite{ms} or Theorems 5.2.2 and 5.3.8  in \cite{GLC}).
\end{proof}

It should be noted that (\ref{Duserge}) gives sufficient information for the
approximation to $u^{\varepsilon }$ in the so-called asymptotic case, i.e.
when the degree $p$ of the approximating polynomials satisfies $%
p>O(\varepsilon ^{-1})$. \ For the pre-asymptotic case, i.e. when $p\leq
O(\varepsilon ^{-1})$, we will need the regularity results provided in the
next section.


\section{Expansion of the solution\label{expansion}}


The solution of (\ref{bvpa})--(\ref{bvpf}) may be decomposed as 
\begin{equation}
u^{\varepsilon }=w^{\varepsilon }+\chi _{BL}u_{BL}^{\varepsilon }+\chi
_{IL}u_{IL}^{\varepsilon }+r^{\varepsilon },  \label{decomp}
\end{equation}%
where $w^{\varepsilon }$ denotes the outer (smooth) part, $%
u_{BL}^{\varepsilon }$ denotes the boundary layer along $\partial \Omega
_{+}\backslash \Sigma $, $u_{IL}^{\varepsilon }$ denotes the interface layer
along $\Sigma $ and $r^{\varepsilon }$ denotes the remainder. The functions $%
\chi _{BL}$,$\chi _{IL}$ denote smooth cut-off functions (see equations (\ref%
{chi_BL}), (\ref{chi_IL}) ahead) in order to account for the fact that the
aforementioned components do not have support in the entire domain $\Omega $.

In order to define the inner (boundary layer) expansion we introduce \emph{%
boundary fitted coordinates} as follows: Let $\left( X(\theta ),Y(\theta
)\right) ,\theta \in \lbrack 0,L]$ be an analytic $L-$periodic
parametrization of $\partial \Omega _{+}\backslash \Sigma $ (by arc length),
such that the normal vector $\left( -Y^{\prime }(\theta ),X^{\prime }(\theta
)\right) $ always points into the domain $\Omega _{+}$. \ Let $%
\kappa_{+}(\theta )$ denote the curvature of $\partial \Omega _{+}\backslash
\Sigma $ and denote by $\mathbb{T}_{L}$ the one-dimensional torus of length $%
L$. By the analyticity of $\partial \Omega $ we have that the functions $X,Y$
and $\kappa $ are analytic. We also let $\rho _{0}>0$ be a fixed constant
satisfying 
\begin{equation}
0<\rho _{0}<\frac{1}{\left\Vert \kappa _{+}\right\Vert _{L^{\infty }\left(
[0,L)\right) }}.  \label{rho_0}
\end{equation}
Then the mapping $\psi :[0,\rho _{0}]\times \mathbb{T}_{L}\rightarrow 
\overline{\Omega }_{+}$ given by 
\begin{equation}
\psi :(\rho ,\theta )\rightarrow \left( X(\theta )-\rho Y^{\prime }(\theta
),Y(\theta )+\rho X^{\prime }(\theta )\right)  \label{psi}
\end{equation}
is real analytic on $[0,\rho _{0}]\times \mathbb{T}_{L}$. \ The function $%
\psi $ maps the rectangle $(0,\rho _{0})\times \lbrack 0,L)$ onto a
half-tubular neighborhood $\Omega _{+}^{0}$ of $\partial
\Omega_{+}\backslash \Sigma $, which may be described as 
\begin{equation}
\Omega _{+}^{0}=\left\{ z-\rho \mathbf{n}_{z}:z\in \partial \Omega
_{+}\backslash \Sigma ,0<\rho <\rho _{0}\right\} ,  \label{Omega_0}
\end{equation}
with $z=z(\theta )=\left( X(\theta ),Y(\theta )\right) $ and $\mathbf{n}_{z}$
the outward unit normal at $z\in \partial \Omega _{+}\backslash \Sigma $.

The interface layer will also be defined in a neighborhood of the interface $%
\Sigma $. Quite analogously, let $\left( X_{\Sigma }(\theta ),Y_{\Sigma
}(\theta )\right) ,\theta \in \lbrack 0,L_{\Sigma }]$ be an analytic $%
L_{\Sigma }-$periodic parametrization of $\Sigma $ (as above), let $\kappa
_{\Sigma }(\theta )$ denote the curvature of $\Sigma $ and denote by $%
\mathbb{T}_{L_{\Sigma }}$ the one-dimensional torus of length $L_{\Sigma }$.
With $\rho _{\Sigma }>0$ a fixed constant satisfying 
\begin{equation}
0<\rho _{\Sigma }<\frac{1}{\left\Vert \kappa _{\Sigma }\right\Vert
_{L^{\infty }\left( [0,L_{\Sigma })\right) }},  \label{rho_sigma}
\end{equation}
we define, analogously to (\ref{Omega_0}), 
\begin{equation}
\Omega _{\Sigma }^{0}=\left\{ z-\rho \mathbf{n}_{\Sigma }:z\in \Sigma
,0<\rho <\rho _{\Sigma }\right\} ,  \label{Omega_sigma}
\end{equation}
with $z=z(\theta )=\left( X_{\Sigma }(\theta ),Y_{\Sigma }(\theta )\right) $
and $\mathbf{n}_{\Sigma }$ the outward unit normal at $z\in \Sigma $.

The smooth cut-off functions $\chi _{BL},\chi _{IL}$ appearing in (\ref%
{decomp}) are defined as follows: Let $\rho _{1},\rho _{2}$ be given
satisfying 
\begin{equation}
0<\rho _{1}<\rho _{0}\;,\;0<\rho _{2}<\rho _{\Sigma }  \label{rho_i}
\end{equation}%
and let $\chi _{BL},\chi _{IL}$ be defined on $\overline{\Omega }_{+}$ via 
\begin{equation}
\chi _{BL}(x)=\left\{ 
\begin{array}{cc}
1 & \text{for }0\leq \text{dist}(x,\partial \Omega _{+}\backslash \Sigma
)\leq \rho _{1}\text{ \ \ \ \ \ \ \ } \\ 
0 & \text{for dist}(x,\partial \Omega _{+}\backslash \Sigma )\geq (\rho
_{1}+\rho _{0})/2%
\end{array}%
\right. ,  \label{chi_BL}
\end{equation}
\begin{equation}
\chi _{IL}(x)=\left\{ 
\begin{array}{cc}
1 & \text{for }0\leq \text{dist}(x,\Sigma )\text{ }\leq \rho _{2}\text{\ \ \
\ \ \ \ } \\ 
0 & \text{for dist}(x,\Sigma )\geq (\rho _{2}+\rho _{\Sigma })/2%
\end{array}%
\right. .  \label{chi_IL}
\end{equation}
The above will be utilized in sections \ref{BL_constr} and \ref{IL_constr}
ahead.

\subsection{Construction and regularity of the outer part \label{smooth}}

We begin by constructing the outer part $w^{\varepsilon }$ in (\ref{decomp}%
). To this end, we expand the solution $u^{\varepsilon }=\left(
u_{+}^{\varepsilon },u_{-}^{\varepsilon }\right) $ as a formal series in
powers of $\varepsilon$, 
\begin{equation}
u_{\pm }^{\varepsilon }\;=u_{0}^{\pm }+\varepsilon u_{1}^{\pm }+\varepsilon
^{2}u_{2}^{\pm }+...  \label{uj}
\end{equation}
and insert it in the differential equations (\ref{bvpa})--(\ref{bvpf}),
equating like powers of $\varepsilon $. This allows us to get expressions
for the functions $u_{j}^{\pm },j=0,1,2,...$. In particular, for $u_{j}^{+}$
we obtain 
\begin{equation}
u_{0}^{+}=f_{+},\;u_{2j}^{+}=\Delta
^{(2j+2)}f_{+},\;u_{2j-1}^{+}=0,\;j=1,2,...  \label{uj+}
\end{equation}
where $\Delta ^{(i)}$ denotes the iterated Laplacian. For $u_{0}^{-}$ we
obtain 
\begin{eqnarray}
-\Delta u_{0}^{-}+u_{0}^{-} &=&f_{-}\text{ in }\Omega _{-},  \label{u0-a} \\
u_{0}^{-} &=&0\text{ on }\partial \Omega _{-}\backslash \Sigma ,
\label{u0-b} \\
\frac{\partial u_{0}^{-}}{\partial \nu } &=&-h\text{ on }\Sigma .
\label{u0-c}
\end{eqnarray}
For $j\geq 1$ we find $u_{2j-1}^{-}=0$ and 
\begin{eqnarray}
-\Delta u_{2j}^{-}+u_{2j}^{-} &=&0\text{ in }\Omega _{-},  \label{uj-a} \\
u_{2j}^{-} &=&0\text{ on }\partial \Omega _{-}\backslash \Sigma ,
\label{uj-b} \\
\frac{\partial u_{2j}^{-}}{\partial \nu } &=&\frac{\partial u_{2j-2}^{+}}{%
\partial \nu }\text{ on }\Sigma .  \label{uj-c}
\end{eqnarray}
Note that $u_{2j}^{-}$ is not explicitly known but is solution of a
Dirichlet-Neumann problem in $\Omega _{-}$. Due to the analyticity
assumption, $u_{2j}^{-}$ is analytic as well (see equation (\ref{eq13})
ahead).

Using the above, we can define the \emph{outer expansion} as 
\begin{equation}
w^{\pm }\equiv w_{M}^{\pm }=\overset{M}{\underset{j=0}{\dsum }}\varepsilon
^{2j}u_{2j}^{\pm },  \label{wM}
\end{equation}
where $M$ is the order of the expansion (i.e. the number of terms that we
will include) and will ultimately be taken to be proportional to $%
1/\varepsilon $ (cf. \cite{melenk}, \cite{ms}). It is not difficult to see
that 
\begin{equation}
\left( -\varepsilon ^{2}\Delta w_{M}^{+}+w_{M}^{+}\right) -f_{+}=\varepsilon
^{2M+2}\Delta ^{(M+1)}f_{+}  \label{eq11}
\end{equation}
and 
\begin{equation}
\left( -\Delta w_{M}^{-}+w_{M}^{-}\right) -f_{-}=0.  \label{eq18}
\end{equation}
Moreover, we have the following theorem.

\begin{theorem}
Let $w_{M}^{\pm }$ be defined by (\ref{wM}). Then there exist positive
constants $K_{1}$ and $C$ depending only on the data of the problem, such
that if $\varepsilon M$ is sufficiently small then
\begin{equation}
\left\Vert D^{\alpha }w_{M}^{+}\right\Vert _{\infty ,\Omega _{+}}\lesssim
K_{1}^{|\alpha |}|\alpha |!\;\forall \;\alpha \in \mathbb{N}_{0}^{2},
\label{eq14a}
\end{equation}
\begin{equation}
\left\Vert w_{M}^{-}\right\Vert _{k,\Omega _{-}}\lesssim C^{k+1}k!\;.
\label{eq14}
\end{equation}
\end{theorem}

\begin{proof} From Theorem 2.2 of \cite{ms1} we have that
\begin{equation*}
\left\Vert D^{\alpha }w_{M}^{+}\right\Vert _{\infty ,\Omega _{+}}\lesssim
K_{1}^{|\alpha |}|\alpha |!\left( 1+\left( 2M\varepsilon K_{2}\right)
^{2M}\right) \;\forall \;\alpha \in \mathbb{N}_{0}^{2},
\end{equation*}
so if $2M\varepsilon K_{2}<1$ we get (\ref{eq14a}) . \ In order to establish
(\ref{eq14}), we first consider $u_{0}^{-}$, which satisfies the
Dirichlet-Neumann problem (\ref{u0-a})--(\ref{u0-c}). Since the data of this
problem are analytic, we have that $u_{0}^{-}$ is also analytic \cite{GLC},
and moreover
\begin{equation*}
\left\vert u_{0}^{-}\right\vert _{k,\Omega _{-}}\leq C^{k+1}k!\;\forall
\;k\in \mathbb{N}_0.
\end{equation*}
Next, we consider $u_{2j}^{-},$ $j=0,1,...,$ defined by (\ref{uj-a})--(\ref{uj-c}), 
with again the data being analytic. \ Casting (\ref{uj-a})--(\ref{uj-c}) 
into a variational formulation, allows us to write
\begin{equation*}
\left\Vert u_{2j}^{-}\right\Vert _{1,\Omega _{-}}^{2}=\underset{\Sigma }{%
\dint }\frac{\partial u_{2j-2}^{+}}{\partial \nu }u_{2j}^{-}\lesssim
\left\Vert \frac{\partial u_{2j-2}^{+}}{\partial \nu }\right\Vert _{0,\Sigma
}\left\Vert u_{2j}^{-}\right\Vert _{1,\Omega _{-}},
\end{equation*}
which, using (\ref{uj+}), gives
\begin{equation}
\left\Vert u_{2j}^{-}\right\Vert _{1,\Omega _{-}}\lesssim \left\Vert \frac{%
\partial u_{2j-2}^{+}}{\partial \nu }\right\Vert _{0,\Sigma }\lesssim
\left\Vert \frac{\partial \left( \Delta ^{2j}f_{+}\right) }{\partial \nu }%
\right\Vert _{0,\Sigma }\lesssim \left\Vert \nabla ^{4j+1}f_{+}\right\Vert
_{\infty ,\Omega _{+}}.  \label{eq12}
\end{equation}
From \cite{GLC} we have that there exists $C\in \mathbb{R}^{+}$ such that 
\begin{equation}
\frac{1}{k!}\left\vert u_{2j}^{-}\right\vert _{k,\Omega _{-}}\leq
C^{k+1}\left\{ \underset{\ell =1}{\overset{k}{\dsum }}\frac{1}{\ell !}%
\left\Vert \frac{\partial u_{2j}^{-}}{\partial \nu }\right\Vert _{\ell +%
\frac{1}{2},\Sigma }+\left\Vert u_{2j}^{-}\right\Vert _{1,\Omega
_{-}}\right\} ,  \label{eq12a}
\end{equation}
and we note that (see eq. (\ref{uj-c})), 
\begin{eqnarray*}
\left\Vert \frac{\partial u_{2j}^{-}}{\partial \nu }\right\Vert _{\ell +%
\frac{1}{2},\Sigma }^{2} &=&\left\Vert \frac{\partial u_{2j-2}^{+}}{\partial
\nu }\right\Vert _{\ell +\frac{1}{2},\Sigma }^{2}=\left\Vert \frac{\partial
\left( \Delta ^{2j}f_{+}\right) }{\partial \nu }\right\Vert _{\ell +\frac{1}{%
2},\Sigma }^{2}\lesssim \left\Vert \Delta ^{2j}f_{+}\right\Vert _{\ell
+1,\Omega _{+}}^{2} \\
&\lesssim &\underset{\left\vert \alpha \right\vert \leq \ell +1}{\dsum }\ 
\underset{\Omega _{+}}{\dint }\left\vert D^{\alpha }\Delta
^{2j}f_{+}\right\vert ^{2}dx\lesssim \underset{\left\vert \alpha \right\vert
\leq \ell +1}{\dsum }\left\Vert D^{\alpha }\Delta ^{2j}f_{+}\right\Vert
_{\infty ,\Omega _{+}}^{2}.
\end{eqnarray*}
Hence, (\ref{eq12a}) becomes (with the aid of (\ref{assf}) and (\ref{eq12}))
\begin{eqnarray*}
\frac{1}{k!}\left\vert u_{2j}^{-}\right\vert _{k,\Omega _{-}} &\leq
&C^{k+1}\left\{ \underset{\ell =1}{\overset{k}{\dsum }}\frac{1}{\ell !}%
\underset{\left\vert \alpha \right\vert \leq \ell +1}{\dsum }\left\Vert
D^{\alpha }\Delta ^{2j}f_{+}\right\Vert _{\infty ,\Omega _{+}}+\left\Vert
\Delta ^{2j}f_{+}\right\Vert _{\infty ,\Sigma }\right\}  \\
&\lesssim &C^{k+1}\left\{ \underset{\ell =1}{\overset{k}{\dsum }}\frac{1}{%
\ell !}\underset{\left\vert \alpha \right\vert \leq \ell +1}{\dsum }\gamma
^{\left\vert \alpha \right\vert +2j}\left( \left\vert \alpha \right\vert
+2j\right) !+\gamma ^{2j}\left( 2j\right) !\right\}  \\
&\lesssim &C^{k+1}\underset{\ell =1}{\overset{k}{\dsum }}\ell ^{2}\gamma
^{\ell +2j}\left( \ell +2j\right) ! \\
&\lesssim &C^{k+1}k^{2}\left( 2j\right) !\underset{\ell =1}{\overset{k}{%
\dsum }}\gamma ^{k+2j}\binom{\ell +2j}{\ell } \\
&\lesssim &C_{1}^{k+1}\left( 2j\right) !\left( 1+\gamma \right)
^{k+2j}\gamma ^{2j} \\
&\lesssim &C_{1}^{k+1}\left( 2j\right) !\gamma _{1}^{2j}
\end{eqnarray*}
for a suitable $C_{1},\gamma _{1}>0.$ \ This shows that $u_{2j}^{-}$ are
analytic and $\forall \;j=0,1,...$ 
\begin{equation}
\left\vert u_{2j}^{-}\right\vert _{k,\Omega _{-}}\lesssim
C_{1}^{k+1}k!\left( 2j\right) !\gamma _{1}^{2j}\;,\;k\in \mathbb{N}.
\label{eq13}
\end{equation}
Thus, from the definition of $w_{M}^{-}$ we have
\begin{eqnarray*}
\left\vert w_{M}^{-}\right\vert _{k,\Omega _{-}} &\leq &\underset{j=0}{%
\overset{M}{\dsum }}\varepsilon ^{2j}\left\vert u_{2j}^{-}\right\vert
_{k,\Omega _{-}}\lesssim C_{1}^{k+1}k!\underset{j=0}{\overset{M}{\dsum }}%
\varepsilon ^{2j}\left( 2j\right) !\gamma _{1}^{2j} \\
&\lesssim &C_{1}^{k+1}k!\underset{j=0}{\overset{M}{\dsum }}\varepsilon
^{2j}\left( 2M\right) ^{2j}\gamma _{1}^{2j}\lesssim C_{1}^{k+1}k!\underset{%
j=0}{\overset{M}{\dsum }}\left( 2\varepsilon M\gamma _{1}\right) ^{2j} \\
&\lesssim &C^{k+1}k!,
\end{eqnarray*}
\emph{provided} $2\varepsilon M\gamma _{1}<1\,\ $(so that the above sum can
be estimated by a converging geometric series). \ Estimate (\ref{eq14})
follows.
\end{proof}

\begin{remark}
The above theorem gives bounds on the smooth (outer) part of the solution to
(\ref{bvpa})--(\ref{bvpf}) under the assumption that $\varepsilon M$ is
sufficiently small. In the complementary case, the asymptotic expansion loses its meaning.
\end{remark}

\subsection{Construction and regularity of the boundary layers along $%
\partial \Omega _{+}\backslash \Sigma \label{BL_constr}$}

Boundary layers are introduced in order to account for the fact that the
function $w_{M}^{+}$ does not satisfy the boundary condition on $\partial
\Omega _{+}\backslash \Sigma $ (cf. (\ref{eq11})). These are precisely the
ones constructed and analyzed in \cite{ms1}, so we will only outline the
procedure and quote the relevant results from \cite{ms1}. The boundary layer
correction $u_{BL}^{\varepsilon }$ of $w_{M}^{+}$ is defined as the solution
of 
\begin{eqnarray}
L_{\varepsilon }u_{BL}^{\varepsilon } &=&0\text{ in }\Omega _{+},
\label{LuBL} \\
u_{BL}^{\varepsilon } &=&-w_{M}^{+}\text{ on }\partial \Omega _{+}\backslash
\Sigma ,  \label{uBL_BC}
\end{eqnarray}
where $L_{\varepsilon }$ is defined as 
\begin{equation}
L_{\varepsilon }u:=-\varepsilon ^{2}\Delta u+u.  \label{operator}
\end{equation}
With $\kappa _{+}(\theta )$ the curvature of $\partial \Omega _{+}\backslash
\Sigma $ we set 
\begin{equation*}
\sigma _{+}(\rho ,\theta )=\frac{1}{1-\kappa _{+}(\theta )\rho },
\end{equation*}
and we have (see, e.g. \cite{AF}) 
\begin{equation*}
\Delta u(\rho ,\theta )=\partial _{\rho }^{2}u-\kappa _{+}(\theta )\sigma
_{+}(\rho ,\theta )\partial _{\rho }u+\sigma _{+}^{2}(\rho ,\theta )\partial
_{\theta }^{2}u+\rho \kappa _{+}^{\prime }(\theta )\sigma _{+}^{3}(\rho
,\theta )\partial _{\theta }u.
\end{equation*}
Introducing the stretched variable $\widehat{\rho }=\rho /\varepsilon $, the
operator $L_{\varepsilon }$ becomes 
\begin{equation}
L_{\varepsilon }=-\partial _{\widehat{\rho }}^{2}+\text{Id}+\varepsilon
\kappa _{+}(\theta )\sigma _{+}(\varepsilon \widehat{\rho },\theta )\partial
_{\widehat{\rho }}-\varepsilon ^{2}\sigma _{+}^{2}(\varepsilon \widehat{\rho 
},\theta )-\varepsilon ^{3}\widehat{\rho }\kappa _{+}^{\prime }(\theta
)\sigma _{+}^{3}(\varepsilon \widehat{\rho },\theta )\partial _{\theta }.
\label{operator_coords}
\end{equation}
Expanding the above in power series of $\varepsilon $, we can formally write 
\begin{equation}
L_{\varepsilon }=\underset{i=0}{\overset{\infty }{\dsum }}\varepsilon
^{i}L_{i},  \label{Lepsilon}
\end{equation}
where the operators $L_{i}$ have the form (see equations (2.12)--(2.14) in 
\cite{ms1}) 
\begin{equation}
L_{0}=-\partial _{\widehat{\rho }}^{2}+\text{Id},\;L_{i}=-\widehat{\rho }%
^{i-1}a_{1}^{i-1}\partial _{\widehat{\rho }}-\widehat{\rho }%
^{i-2}a_{2}^{i-2}\partial _{\theta }^{2}-\widehat{\rho }^{i-2}a_{3}^{i-3}%
\partial _{\theta },\;i\geq 1,  \label{2.12}
\end{equation}
and the coefficients $a_{j}^{i}$ are given by 
\begin{equation}
a_{1}^{i}=-\left[ \kappa _{+}(\theta )\right] ^{i+1},\;a_{2}^{i}=\left(
i+1\right) \left[ \kappa _{+}(\theta )\right] ^{i},\;a_{3}^{i}=\frac{%
(i+1)(i+2)}{2}\left[ \kappa _{+}(\theta )\right] ^{i}\kappa _{+}^{\prime
}(\theta ),\;i\in \mathbb{N}_{0},  \label{2.13}
\end{equation}
\begin{equation}
a_{1}^{i}=a_{2}^{i}=a_{3}^{i}=0\text{ for }i<0.  \label{2.14}
\end{equation}
We next make the formal ansatz 
\begin{equation*}
u_{BL}^{\varepsilon }=\underset{i=0}{\overset{\infty }{\dsum }}\varepsilon
^{i}\widehat{U}_{i}(\widehat{\rho },\theta ),
\end{equation*}
and insert it into (\ref{LuBL}). This yields 
\begin{equation}
\underset{i=0}{\overset{\infty }{\dsum }}\varepsilon ^{i}\underset{j=0}{%
\overset{i}{\dsum }}L_{j}\widehat{U}_{i-j}=0,  \label{sumLepsilon}
\end{equation}
allowing us to find the following problem for the functions $\widehat{U}_{i}(%
\widehat{\rho },\theta ),i=0,1,2,...$: 
\begin{equation}
-\partial _{\widehat{\rho }}^{2}\widehat{U}_{i}+\widehat{U}_{i}=\widehat{F}%
_{i}=:\widehat{F}_{i}^{1}+\widehat{F}_{i}^{2}+\widehat{F}_{i}^{3},
\label{eq2.15}
\end{equation}
\begin{equation}
\widehat{F}_{i}^{1}=\underset{k=0}{\overset{i-1}{\dsum }}\widehat{\rho }%
^{k}a_{1}^{k}\partial _{\widehat{\rho }}\widehat{U}_{i-1-k}\;,\;\widehat{F}%
_{i}^{2}=\underset{k=0}{\overset{i-2}{\dsum }}\widehat{\rho }%
^{k}a_{2}^{k}\partial _{\theta }^{2}\widehat{U}_{i-2-k}\;,\;\widehat{F}%
_{i}^{3}=\underset{k=0}{\overset{i-3}{\dsum }}\widehat{\rho }%
^{k+1}a_{3}^{k}\partial _{\theta }\widehat{U}_{i-3-k},  \label{eq2.16}
\end{equation}
where empty sums are assumed to be zero. (See, also, equations
(2.15)--(2.16) in \cite{ms1}). The above are supplemented with boundary
conditions 
\begin{eqnarray*}
\widehat{U}_{i} &\rightarrow &0\text{ as }\rho \rightarrow \infty , \\
\left[ \widehat{U}_{i}\right] _{\partial \Omega \backslash \Sigma }
&=&\left\{ 
\begin{array}{ccc}
-\left[ f\right] _{\partial \Omega \backslash \Sigma } & \text{if} & i=0, \\ 
-\left[ \Delta ^{(i/2)}f\right] _{\partial \Omega \backslash \Sigma } & 
\text{if} & i\in \mathbb{N}\text{ is even,} \\ 
0 & \text{if} & i\in \mathbb{N}\text{ is odd.}%
\end{array}%
\right.
\end{eqnarray*}
The boundary layer (inner) expansion in (\ref{decomp}) is then defined as 
\begin{equation}
u_{BL}^{\varepsilon }\equiv u_{BL}^{M}(\rho ,\theta )=\underset{j=0}{\overset%
{2M+1}{\dsum }}\varepsilon ^{j}\widehat{U}_{j}(\widehat{\rho },\theta )=%
\underset{j=0}{\overset{2M+1}{\dsum }}\varepsilon ^{j}\widehat{U}_{j}(\rho
/\varepsilon ,\theta ),  \label{uBL_M}
\end{equation}
and by construction, it satisfies the boundary condition 
\begin{equation*}
\left[ u_{BL}^{M}\right] _{\partial \Omega \backslash \Sigma }=-\underset{i=0%
}{\overset{2M+1}{\dsum }}\varepsilon ^{2i}\left[ \Delta ^{(i)}f\right]
_{\partial \Omega \backslash \Sigma }.
\end{equation*}
By Theorem 2.2 of \cite{ms1} we have that for every $\alpha \in \lbrack 0,1)$
and all $p,m\in \mathbb{N}_{0}$, 
\begin{equation}
\left\vert \partial _{\rho }^{p}\partial _{\theta }^{m}u_{BL}^{M}(\rho
,\theta )\right\vert \lesssim \left( 1+\left( \frac{\varepsilon (2M+1)K_{2}}{%
1-\alpha }\right) ^{2M+1}\right) m!K_{1}^{m+p}\varepsilon ^{-p}e^{-\alpha
\rho /\varepsilon },\;  \label{uBL_bound}
\end{equation}
for $\theta \in \mathbb{T}_{L},\rho \in \lbrack 0,\rho _{0}],$ with $%
K_{1},K_{2}>0$ independent of $\varepsilon ,p$ and $m$. \ Moreover, by Lemma
2.12 of \cite{ms1}, there exist constants $K,\Theta >0$ independent of $%
\varepsilon $ such that 
\begin{equation}
\left\vert L_{\varepsilon }u_{BL}^{M}\left( \rho ,\theta \right) \right\vert
\lesssim K^{2M+2}\left( \varepsilon (2M+2)+\left\vert \rho \right\vert
\right) ^{2M+2}e^{-\rho /\varepsilon }\;\forall \;(\rho ,\theta )\in B_{\rho
_{0}}(0)\times S(\Theta ),  \label{LuBL_bound}
\end{equation}
where $B_{\delta }(z)$ denotes the (open) disc in the complex plane of
radius $\delta $ centered at $z$, and 
\begin{equation}
S(\Theta )=\left\{ \theta \in \mathbb{C}:\func{Im}(\theta )<\Theta \right\} .
\label{Stheta}
\end{equation}

\subsection{Construction and regularity of the interface layer on $\Sigma 
\label{IL_constr}$}

For a function $\omega =\left( \omega _{+},\omega _{-}\right) $ we denote
the jump $\left[ [\omega ]\right] _{\Sigma }$ on $\Sigma $ as 
\begin{equation}
\lbrack \left[ \omega ]\right] _{\Sigma }:=\left. \left( \omega _{+}\right)
\right\vert _{\Sigma }-\left. \left( \omega _{-}\right) \right\vert _{\Sigma
}.  \label{eq26}
\end{equation}%
We define the function $v_{I}(\rho ,\theta ):=\left(
v_{I}^{-},v_{I}^{+}\right) $ as the solution of the following problem: 
\begin{equation}
\left\{ 
\begin{array}{c}
\begin{array}{c}
\begin{array}{c}
-\varepsilon ^{2}\Delta v_{I}^{+}+v_{I}^{+}=0\text{ in }\Omega _{+} \\ 
-\Delta v_{I}^{-}+v_{I}^{-}=0\text{ in }\Omega _{-}%
\end{array}
\\ 
\lbrack \left[ v_{I}\right] ]_{\Sigma }=-\underset{j=0}{\overset{\infty }{%
\dsum }}\varepsilon ^{2j}[\left[ u_{2j}\right] ]_{\Sigma }%
\end{array}
\\ 
\left. \left( \varepsilon ^{2}\frac{\partial v_{I}^{+}}{\partial \rho }-%
\frac{\partial v_{I}^{-}}{\partial \rho }\right) \right\vert _{\Sigma }=-%
\underset{j=0}{\overset{\infty }{\dsum }}\varepsilon ^{2j}\left. \left(
\varepsilon ^{2}\frac{\partial u_{2I}^{+}}{\partial \rho }-\frac{\partial
u_{2j}^{-}}{\partial \rho }\right) \right\vert _{\Sigma }%
\end{array}%
\right. .  \label{eq27}
\end{equation}%
With $\widehat{\rho }=\rho /\varepsilon $ as before, we write 
\begin{equation*}
\widehat{v}_{I}^{+}(\rho ,\theta )=v_{I}^{+}\left( \widehat{\rho },\theta
\right) ,
\end{equation*}%
and problem (\ref{eq27}) becomes 
\begin{equation}
\left\{ 
\begin{array}{c}
\begin{array}{c}
\left( -\partial _{\widehat{\rho }}^{2}+\text{Id}\right) \widehat{v}%
_{I}^{+}+\varepsilon \kappa _{+}(\theta )\sigma (\varepsilon \widehat{\rho }%
,\theta )\partial _{\widehat{\rho }}\widehat{v}_{I}^{+}-\varepsilon
^{2}\sigma ^{2}(\varepsilon \widehat{\rho },\theta )\widehat{v}_{I}^{+}- \\ 
-\varepsilon ^{3}\widehat{\rho }\kappa _{+}^{\prime }(\theta )\sigma
^{3}(\varepsilon \widehat{\rho },\theta )\partial _{\theta }\widehat{v}%
_{I}^{+}=0\text{ in }\Omega _{+}%
\end{array}
\\ 
-\Delta v_{I}^{-}+v_{I}^{-}=0\text{ in }\Omega _{-} \\ 
\left( \widehat{v}_{I}^{+}-v_{I}^{-}\right) |_{\Sigma }=-\underset{j=0}{%
\overset{\infty }{\dsum }}\varepsilon ^{2j}\left[ [u_{2j}]\right] _{\Sigma }
\\ 
\left. \left( \varepsilon \frac{\partial \widehat{v}_{I}^{+}}{\partial 
\widehat{\rho }}-\frac{\partial v_{I}^{-}}{\partial \rho }\right)
\right\vert _{\Sigma }=-\underset{j=0}{\overset{\infty }{\dsum }}\varepsilon
^{2j}\left. \left( \varepsilon ^{2}\frac{\partial u_{2j}^{+}}{\partial \rho }%
-\frac{\partial u_{2j}^{-}}{\partial \rho }\right) \right\vert _{\Sigma }%
\end{array}%
\right. .  \label{eq28}
\end{equation}%
Now, we write 
\begin{equation}
\widehat{v}_{I}^{+}=\underset{j=0}{\overset{\infty }{\dsum }}\varepsilon ^{j}%
\widehat{V}_{j}^{+}\;,\;v_{I}^{-}=\underset{j=0}{\overset{\infty }{\dsum }}%
\varepsilon ^{j}V_{j}^{-},  \label{vhat+v-}
\end{equation}%
and insert it in (\ref{eq28}) equating like powers of $\varepsilon $, to get
(utilizing again the expansion (\ref{Lepsilon})) 
\begin{equation}
\left\{ 
\begin{array}{c}
-\partial _{\widehat{\rho }}^{2}\widehat{V}_{j}^{+}+\widehat{V}_{j}^{+}=%
\widehat{F}_{j}^{1}+\widehat{F}_{j}^{2}+\widehat{F}_{j}^{3}\text{ in }%
\mathbb{R}_{+}\;\forall \;j\geq 0 \\ 
-\Delta V_{j}^{-}+V_{j}^{-}=0\text{ in }\Omega _{-}\;\forall \;j\geq 0 \\ 
\left( \widehat{V}_{2j}^{+}-V_{2j}^{-}\right) =-\left(
u_{2j}^{+}-u_{2j}^{-}\right) \text{ on }\Sigma \;\forall \;j\geq 0 \\ 
\left( \widehat{V}_{2j+1}^{+}-V_{2j+1}^{-}\right) =0\text{ on }\Sigma
\;\forall \;j\geq 0 \\ 
-\frac{\partial V_{0}^{-}}{\partial \rho }=\frac{\partial u_{0}^{-}}{%
\partial \rho }\text{ on }\Sigma \\ 
\left( \frac{\partial }{\partial \rho }V_{2j}^{-}-\frac{\partial }{\partial 
\widehat{\rho }}\widehat{V}_{2j-1}^{+}\right) =-\left( \frac{\partial }{%
\partial \rho }u_{2j}^{-}-\frac{\partial }{\partial \rho }%
u_{2j-2}^{+}\right) \text{ on }\Sigma \;\forall \;j\geq 1 \\ 
\left( \frac{\partial }{\partial \rho }V_{2j+1}^{-}-\frac{\partial }{%
\partial \widehat{\rho }}\widehat{V}_{2j}^{+}\right) =0\text{ on }\Sigma
\;\forall \;j\geq 0%
\end{array}%
\right. ,  \label{eq29}
\end{equation}%
with $\widehat{F}_{j}^{1},\widehat{F}_{j}^{2},\widehat{F}_{j}^{3}$ given by (%
\ref{eq2.16}) but with $\widehat{U}$ replaced by $\widehat{V}^{+}.$ So for $%
j=0$, we have 
\begin{equation}
\left\{ 
\begin{array}{c}
-\Delta V_{0}^{-}+V_{0}^{-}=0\text{ in }\Omega _{-} \\ 
-\frac{\partial V_{0}^{-}}{\partial \rho }=\frac{\partial u_{0}^{-}}{%
\partial \rho }\text{ on }\Sigma \\ 
V_{0}^{-}=0\text{ on }\partial \Omega _{-}\backslash \Sigma%
\end{array}%
\right. ,  \label{V0minus}
\end{equation}%
\begin{equation}
\left\{ 
\begin{array}{c}
-\partial _{\widehat{\rho }}^{2}\widehat{V}_{0}^{+}+\widehat{V}_{0}^{+}=0%
\text{ in }\mathbb{R}_{+} \\ 
\widehat{V}_{0}^{+}=V_{0}^{-}-\left( u_{0}^{+}-u_{0}^{-}\right) \text{ on }%
\Sigma%
\end{array}%
\right. ,  \label{V0plus}
\end{equation}%
\begin{equation}
\left\{ 
\begin{array}{c}
-\Delta V_{1}^{-}+V_{1}^{-}=0\text{ in }\Omega _{-} \\ 
\frac{\partial }{\partial \rho }V_{1}^{-}=\frac{\partial }{\partial \widehat{%
\rho }}\widehat{V}_{0}^{+}\text{ on }\Sigma \\ 
V_{1}^{-}=0\text{ on }\partial \Omega _{-}\backslash \Sigma%
\end{array}%
\right. ,  \label{V1minus}
\end{equation}%
\begin{equation}
\left\{ 
\begin{array}{c}
-\partial _{\widehat{\rho }}^{2}\widehat{V}_{1}^{+}+\widehat{V}_{1}^{+}=%
\widehat{V}_{0}^{+}\text{ in }\mathbb{R}_{+} \\ 
\widehat{V}_{1}^{+}=V_{1}^{-}\text{ on }\Sigma%
\end{array}%
\right. .  \label{V1plus}
\end{equation}%
In general, for $j\geq 0$ odd we have 
\begin{equation}
\left\{ 
\begin{array}{c}
-\Delta V_{2j+1}^{-}+V_{2j+1}^{-}=0\text{ in }\Omega _{-} \\ 
\frac{\partial }{\partial \rho }V_{2j+1}^{-}=\frac{\partial }{\partial 
\widehat{\rho }}\widehat{V}_{2j}^{+}\text{ on }\Sigma \\ 
V_{2j+1}^{-}=0\text{ on }\partial \Omega _{-}\backslash \Sigma%
\end{array}%
\right. ,  \label{Vjminus_odd}
\end{equation}%
\begin{equation}
\left\{ 
\begin{array}{c}
-\partial _{\widehat{\rho }}^{2}\widehat{V}_{2j+1}^{+}+\widehat{V}%
_{2j+1}^{+}=\widehat{F}_{2j+1}^{1}+\widehat{F}_{2j+1}^{2}+\widehat{F}%
_{2j+1}^{3}\text{ in }\mathbb{R}_{+} \\ 
\widehat{V}_{2j+1}^{+}=V_{2j+1}^{-}\text{ on }\Sigma%
\end{array}%
\right.  \label{Vjplus_odd}
\end{equation}%
and for $j\geq 0$ even we have 
\begin{equation}
\left\{ 
\begin{array}{c}
-\Delta V_{2j}^{-}+V_{2j}^{-}=0\text{ in }\Omega _{-} \\ 
\left( \frac{\partial }{\partial \rho }V_{2j}^{-}-\frac{\partial }{\partial 
\widehat{\rho }}\widehat{V}_{2j-1}^{+}\right) =-\left( \frac{\partial }{%
\partial \rho }u_{2j}^{-}-\frac{\partial }{\partial \rho }%
u_{2j-2}^{+}\right) \text{ on }\Sigma \\ 
V_{2j}^{-}=0\text{ on }\partial \Omega _{-}\backslash \Sigma%
\end{array}%
\right. ,  \label{Vjminus_even}
\end{equation}%
\begin{equation}
\left\{ 
\begin{array}{c}
-\partial _{\widehat{\rho }}^{2}\widehat{V}_{2j}^{+}+\widehat{V}_{2j}^{+}=%
\widehat{F}_{2j}^{1}+\widehat{F}_{2j}^{2}+\widehat{F}_{2j}^{3}\text{ in }%
\mathbb{R}_{+}, \\ 
\widehat{V}_{2j}^{+}=V_{2j}^{-}-\left( u_{2j}^{+}-u_{2j}^{-}\right) \text{
on }\Sigma%
\end{array}%
\right. .  \label{Vjplus_even}
\end{equation}

The regularity of the functions $V_{j}^{-},\widehat{V}_{j}^{+}$ is given by
Theorem \ref{thm_Vij} below. For its proof, we will need the following lemma.

\begin{lemma} \label{lemmaMS}
Let $U_{j}\left( \widehat{\rho },\theta \right) ,j=0,1,2,...,$ be the
solutions to%
\begin{equation}
\left. 
\begin{array}{c}
-\partial _{\widehat{\rho }}^{2}U_{j}+U_{j}=F_{j}\left( \widehat{\rho }%
,\theta \right) \text{ in }\mathbb{R}_{+} \\ 
U_{j}=G_{j}\left( \theta \right) \text{ on }\Sigma 
\end{array}%
\right\} ,  \label{MSprob}
\end{equation}
where $F_{j}\left( \widehat{\rho },\theta \right) =F_{j}^{1}\left( \widehat{%
\rho },\theta \right) +F_{j}^{2}\left( \widehat{\rho },\theta \right)
+F_{j}^{2}\left( \widehat{\rho },\theta \right) $ is given by (\ref{eq2.15}%
)--(\ref{eq2.16}) and $G_{j}$ satisfy 
\begin{equation}
\left\vert G_{j}\left( \theta \right) \right\vert \leq C_{G}\gamma
_{G}^{j}j^{j},  \label{MSdata}
\end{equation}%
for some positive constants $C_{G},\gamma _{G}$ depending only on the data. 
Then, there exist positive constants $\Theta, C_{U},\gamma _{U},$ depending only
 on the data, such that 
\begin{equation}
\left\vert U_{j}\left( \widehat{\rho },\theta \right) \right\vert \leq
C_{U}\gamma _{U}^{j}\left( 1+j+\widehat{\rho }\right) ^{j}e^{-\widehat{\rho }%
}\;\forall \;\left( \widehat{\rho },\theta \right) \in \mathbb{R}_{+}\times
S\left( \Theta \right),  \label{MSbound1}
\end{equation}
where $S(\Theta )$ is given by (\ref{Stheta}). 
Moreover, for any $\alpha \in [0,1)$ there exists $K\in \mathbb{R}_{+}$
depending only on the data, such that
\begin{equation}
\left\vert U_{j}(\widehat{\rho },\theta )\right\vert \lesssim
K^{j}j^{j}\left( 1-\alpha \right) ^{-j}e^{-\alpha \widehat{\rho }},
\label{MSbound3}
\end{equation}%
and%
\begin{equation}
\left\vert \partial _{\rho }^{p}\partial _{\theta }^{q}U_{j}\left( \rho
/\varepsilon ,\theta \right) \right\vert \lesssim \varepsilon
^{-p}e^{(1-\alpha )p}\left( p+1\right) ^{1/2}q!\left( 2/\Theta \right)
^{q}\gamma _{U}^{j}j^{j}\left( 1-\alpha \right) ^{-j}e^{-\rho /\varepsilon
}\;\forall \;p,q\in \mathbb{N}_{0}.  \label{MSbound2}
\end{equation}%
\end{lemma}%
\begin{proof} This is essentially a combination of Lemmas 2.9 and 2.11 in \cite{ms1}.
Estimate (\ref{MSbound1}) follows directly from Lemma 2.11 in \cite{ms1}, while (\ref{MSbound3})
follows from (\ref{MSbound1}) and Lemma 2.8 in \cite{ms1}. Finally, (\ref{MSbound2}) follows from Cauchy's 
integral formula, in exactly the same way as in the proof of (2.24) in \cite{ms1}.
%
%
%
\end{proof}

\begin{theorem}
\label{thm_Vij} Let $V_{j}^{-}$ satisfy (\ref{Vjminus_odd}), 
(\ref{Vjminus_even}) and $\widehat{V}_{j}^{+}$ satisfy (\ref{Vjplus_odd}), 
(\ref{Vjplus_even}) . Then there exist constants $C,\gamma ,K,\Theta >0$
depending only on the data, such that 
\begin{equation}
\left\vert V_{j}^{-}\right\vert _{k,\Omega _{-}}\lesssim
k!C^{k+1}j^{j}\gamma ^{j},  \label{Vjminus_bound}
\end{equation}
while for $\theta \in \mathbb{T}_{L_{\Sigma }},\rho \in \lbrack 0,\rho
_{\Sigma }],\alpha \in \lbrack 0,1),$ 
\begin{equation}
\left\vert \widehat{V}_{j}^{+}(\rho /\varepsilon ,\theta )\right\vert
\lesssim K^{j}j^{j}\left( 1-\alpha \right) ^{-j}e^{-\alpha \rho /\varepsilon
},  \label{Vjplus_bound}
\end{equation}
and
\begin{equation}
\left\vert \partial _{\rho }^{p}\partial _{\theta }^{q}\widehat{V}%
_{j}^{+}(\rho /\varepsilon ,\theta )\right\vert \lesssim \varepsilon
^{-p}e^{(1-\alpha )p}\left( p+1\right) ^{1/2}q!\left( 2/\Theta \right)
^{q}K^{j}j^{j}\left( 1-\alpha \right) ^{-j}e^{-\alpha \rho /\varepsilon },
\label{Vjplus_bound_derivatives}
\end{equation}
for $p,q\in \mathbb{N}_{0}$ and $\theta \in S(\Theta )$ given by (\ref{Stheta}).
\end{theorem}%
\begin{proof} The proof is by induction on $j$. First we note that estimates (%
\ref{Vjplus_bound}), (\ref{Vjplus_bound_derivatives}) follow from Lemma \ref%
{lemmaMS}, provided we show that (\ref{MSdata}) is satisfied, i.e. on $%
\Sigma $ the functions $\widehat{V}_{j}^{+}$ are bounded by $Cj^{j}\gamma
^{j}$ for suitable constants $C,\gamma >0$. This will be verified during our induction argument; in fact it will be
the only thing we will show for $\widehat{V}_{j}^{+}$, with the
understanding that an application of Lemma \ref{lemmaMS} gives the desired result.

For $j=0$ we see from the variational formulation of (\ref{V0minus}) that $%
\left\Vert V_{0}^{-}\right\Vert _{1,\Omega _{-}}\lesssim \left\Vert
u_{0}^{-}\right\Vert _{1,\Omega _{-}}$, hence by (\ref{eq13}) and \cite{GLC} 
\begin{eqnarray*}
\frac{1}{k!}\left\vert V_{0}^{-}\right\vert _{k,\Omega _{-}} &\leq
&C^{k+1}\left\{ \underset{\ell =0}{\overset{k-2}{\dsum }}\frac{1}{\ell !}%
\left\Vert \frac{\partial u_{0}^{-}}{\partial y}\right\Vert _{\ell +\frac{1}{%
2},\Sigma }+\left\Vert V_{0}^{-}\right\Vert _{1,\Omega _{-}}\right\} \\
&\lesssim &C^{k+1}\left\{ \underset{\ell =1}{\overset{k}{\dsum }}\frac{1}{%
\ell !}\left\Vert u_{0}^{-}\right\Vert _{\ell +1,\Sigma }+\left\Vert
u_{0}^{-}\right\Vert _{1,\Omega _{-}}\right\} \\
&\lesssim &C^{k+1}\left\{ \underset{\ell =1}{\overset{k}{\dsum }}\frac{1}{%
\ell !}C_{1}^{\ell +1}(\ell +1)!\right\} \\
&\lesssim &C^{k+1}.
\end{eqnarray*}
Next, for $\widehat{V}_{0}^{+}$ we see from (\ref{V0plus}) that $\widehat{V}%
_{0}^{+}(\rho ,\theta )=G_{0}\left( \theta \right) e^{-\rho /\varepsilon }$
for some function $G_{0}\left( \theta \right) $ that depends on $%
V_{0}^{-},u_{0}^{+},u_{0}^{-}.$ By the above, (\ref{uj+}) and (\ref{eq13}),
we have that in the case $j=0$, the boundary data for $\widehat{V}_{j}^{+} $
is bounded by $Cj^{j}\gamma ^{j}$ for suitable constants $C,\gamma >0$,
hence by Lemma \ref{lemmaMS} the bounds (\ref{Vjplus_bound}), (\ref{Vjplus_bound_derivatives}) hold for 
$\widehat{V}_{0}^{+}$.

Now, from the variational formulation of (\ref{V1minus}) and the fact that $%
V_{0}^{+}(\hat{\rho},\theta )=G_{0}(\theta )e^{-\hat{\rho}},$ we have $\frac{%
\partial }{\partial \hat{\rho}}V_{0}^{+}(\hat{\rho},\theta )=-G_{0}(\theta
)e^{-\hat{\rho}}$ and then $\frac{\partial }{\partial \hat{\rho}}%
V_{0}^{+}(0,\theta )=-G_{0}(\theta ),$ hence

\begin{equation*}
\int_{\Omega _{-}}(\nabla V_{1}^{-}\cdot \nabla
V+V_{1}^{-}V)\,dx=-\int_{\Sigma }G_{0}V\,dx\;\forall \;V\in H_{\ast
}^{1}(\Omega _{-}),
\end{equation*}
where 
\begin{equation}
H_{\ast }^{1}(\Omega _{-})=\left\{ u\in H^{1}\left( \Omega _{-}\right)
:\left. u\right\vert _{\partial \Omega _{-}\backslash \Sigma }=0\right\} .
\label{H1*}
\end{equation}
Thus, 
\begin{eqnarray*}
\left\Vert V_{1}^{-}\right\Vert _{1,\Omega _{-}} &\lesssim & \left\Vert
G_{0}\right\Vert _{0,\Sigma }= \left\Vert \widehat{V}_{0}^{+}\right\Vert
_{0,\Sigma }=\left\Vert V_{0}^{-}-(u_{0}^{+}-u_{0}^{-})\right\Vert
_{0,\Sigma } \\
&\leq &\left\Vert V_{0}^{-}\right\Vert _{0,\Sigma }+\left\Vert
u_{0}^{+}\right\Vert _{0,\Sigma }+\left\Vert u_{0}^{-}\right\Vert _{0,\Sigma
}\leq C_{1}\in \mathbb{R}^{+}.
\end{eqnarray*}
From \cite{GLC} and the above result, we get 
\begin{eqnarray*}
\frac{1}{k!}\left\vert V_{1}^{-}\right\vert _{k,\Omega _{-}} &\leq
&C^{k+1}\left\{ \underset{\ell =0}{\overset{k-2}{\dsum }}\frac{1}{\ell !}%
\left\Vert \frac{\partial \widehat{V}_{0}^{+}}{\partial \rho }\right\Vert
_{\ell +\frac{1}{2},\Sigma }+\left\Vert V_{1}^{-}\right\Vert _{1,\Omega
_{-}}\right\} \\
&\lesssim &C^{k+1}\left\{ \underset{\ell =0}{\overset{k-2}{\dsum }}\frac{1}{%
\ell !}\left\Vert G_{0}\right\Vert _{\ell +\frac{1}{2},\Sigma }+C_{1}\right\}
\\
&\lesssim &C^{k+1}\left\{ \underset{\ell =0}{\overset{k-2}{\dsum }}\frac{1}{%
\ell !}\left( \left\Vert V_{0}^{-}\right\Vert _{\ell +\frac{1}{2},\Sigma
}+\left\Vert u_{0}^{+}\right\Vert _{\ell +\frac{1}{2},\Sigma }+\left\Vert
u_{0}^{-}\right\Vert _{\ell +\frac{1}{2},\Sigma }\right) +C_{1}\right\} \\
&\lesssim &C^{k+1}\left\{ \underset{\ell =1}{\overset{k}{\dsum }}\frac{1}{%
\ell !}(\ell +1)!C^{\ell +1}+C_{1}\right\} ,
\end{eqnarray*}
which leads to 
\begin{equation}
\left\vert V_{1}^{-}\right\vert _{k,\Omega _{-}}\lesssim C^{k+1}k!.
\label{V1minus_k_norm}
\end{equation}
In an analogous way as $\widehat{V}_{0}^{+}(\widehat{\rho },\theta
)=G_{0}\left( \theta \right) e^{-\widehat{\rho }}$, we find that $\widehat{V}%
_{1}^{+}(\widehat{\rho },\theta )=-\frac{\widehat{\rho }}{2}a_{1}^{0}(\theta
)e^{-\widehat{\rho }}+G_{1}\left( \theta \right) e^{-\widehat{\rho }},$ for
some function $G_{1}\left( \theta \right) $ that depends on $V_{1}^{-},$
hence in view of (\ref{V1minus_k_norm}) we see that the boundary data for $%
\widehat{V}_{1}^{+}$ satisfy the appropriate bound. As a result, (\ref%
{Vjplus_bound})--(\ref{Vjplus_bound_derivatives}) hold for $\widehat{V}%
_{1}^{+}$ as well.

So, we assume that (\ref{Vjminus_bound})--(\ref{Vjplus_bound_derivatives})
hold for $j$ and we will establish them for $j+1$. 

\noindent \textbf{The case of odd $j$:} \ If $j$ is odd, then $j+1$ is even
and we would like to establish bounds for $V_{2s}^{-}$ and $\widehat{V}%
_{2s}^{+}$ (with $2s=j+1$), which satisfy (\ref{Vjminus_even}), (\ref%
{Vjplus_even}) respectively. First, for $V_{2s}^{-}$ we see from the
variational formulation of (\ref{Vjminus_even}) that 
\begin{equation*}
\left\Vert V_{2s}^{-}\right\Vert _{1,\Omega _{-}}\lesssim \left\Vert \frac{%
\partial }{\partial \widehat{\rho }}\widehat{V}_{2s-1}^{+}\right\Vert
_{0,\Sigma }+\left\Vert \frac{\partial }{\partial \rho }u_{2s}^{-}\right\Vert _{0,\Sigma }+\left\Vert \frac{\partial }{\partial \rho }%
u_{2s-2}^{+}\right\Vert _{0,\Sigma },
\end{equation*}%
hence by (\ref{uj+}), (\ref{eq13}), a trace theorem and the induction
hypothesis, we have 
\begin{equation*}
\left\Vert V_{2s}^{-}\right\Vert _{1,\Omega _{-}}\lesssim
K^{2s-1}(2s-1)^{2s-1}+C_{u^{-}}\gamma _{u^{-}}^{2s}(2s)!+C_{f_{+}}\gamma
_{f_{+}}^{2s}(2s)!\lesssim C(2s)!\gamma ^{2s},
\end{equation*}%
for suitable constants $C,\gamma >0$ independent of $s$. Therefore, from 
\cite{GLC} we obtain for $k\geq 2$ 
\begin{eqnarray}
\frac{1}{k!}\left\vert V_{2s}^{-}\right\vert _{k,\Omega _{-}} &\leq
&C^{k+1}\left\{ \underset{\ell =0}{\overset{k-2}{\dsum }}\frac{1}{\ell !}%
\left( \left\Vert \frac{\partial }{\partial \widehat{\rho }}\widehat{V}%
_{2s-1}^{+}\right\Vert _{\ell +\frac{1}{2},\Sigma }+\left\Vert \frac{%
\partial }{\partial \rho }u_{2s}^{-}\right\Vert _{\ell +\frac{1}{2},\Sigma
}+\left\Vert \frac{\partial }{\partial \rho }u_{2s-2}^{+}\right\Vert _{\ell +%
\frac{1}{2},\Sigma }\right) +\left\Vert V_{2s}^{-}\right\Vert _{1,\Omega
_{-}}\right\}   \notag \\
&\lesssim &C^{k+1}\left\{ \underset{\ell =0}{\overset{k-2}{\dsum }}\frac{1}{%
\ell !}\left( \left\Vert \widehat{V}_{2s-1}^{+}\right\Vert _{\ell +\frac{1}{2%
},\Sigma }+\left\Vert u_{2s}^{-}\right\Vert _{\ell +2,\Omega
_{-}}+\left\Vert u_{2s-2}^{+}\right\Vert _{\ell +2,\Omega _{+}}\right)
+(2s)!\gamma ^{2s}\right\} .  \label{V2sminus_knorm}
\end{eqnarray}%
Now, $u_{2s-2}^{+}=$ $\Delta ^{(2s)}f_{+}$ (see eq. (\ref{uj+})), hence
using (\ref{assf}) we have%
\begin{eqnarray*}
\left\Vert u_{2s-2}^{+}\right\Vert _{\ell +1,\Omega _{+}} &=&\left\Vert
\Delta ^{(2s)}f^{+}\right\Vert _{\ell +1,\Omega _{+}}\lesssim \underset{\ell
=1}{\overset{k}{\dsum }}\frac{1}{\ell !}\underset{\left\vert \alpha
\right\vert \leq \ell +1}{\dsum }\left\Vert D^{\alpha }\Delta
^{(2s)}f^{+}\right\Vert _{\infty ,\Omega _{+}} \\
&\lesssim &\underset{\ell =1}{\overset{k}{\dsum }}\frac{1}{\ell !}\underset{%
\left\vert \alpha \right\vert \leq \ell +1}{\dsum }\gamma _{f}^{\left\vert
\alpha \right\vert +2s}\left( \left\vert \alpha \right\vert +2s\right)
!\lesssim \underset{\ell =1}{\overset{k}{\dsum }}\ell ^{2}\gamma _{f}^{\ell
+2s}\left( \ell +2s\right) ! \\
&\lesssim &k^{2}\left( 2s\right) !\underset{\ell =1}{\overset{k}{\dsum }}%
\gamma _{f}^{k+2s}\binom{\ell +2s}{\ell }\lesssim \left( 2s\right) !\left(
1+\gamma _{f}\right) ^{k+2s}\gamma _{f}^{2s} \\
&\lesssim &\left( 2s\right) !\widetilde{\gamma }^{2s},
\end{eqnarray*}%
for suitable $\widetilde{\gamma }>0$ independent of $\varepsilon $. Also, by
(\ref{Vjplus_even}) we get%
\begin{eqnarray*}
\left\Vert \widehat{V}_{2s-1}^{+}\right\Vert _{\ell +\frac{1}{2},\Sigma }
&=&\left\Vert V_{2s-1}^{-}-\left( u_{2s-1}^{+}-u_{2s-1}^{-}\right)
\right\Vert _{\ell +\frac{1}{2},\Sigma } \\
&\leq &\left\Vert V_{2s-1}^{-}\right\Vert _{\ell +1,\Omega _{-}}+\left\Vert
u_{2s-1}^{+}\right\Vert _{\ell +1,\Omega _{+}}+\left\Vert
u_{2s-1}^{-}\right\Vert _{\ell +1,\Omega _{-}}.
\end{eqnarray*}%
Equation (\ref{uj+}) gives $u_{2s-1}^{+}=0,$ and by (\ref{uj+}), (\ref{eq13}%
) and the induction hypothesis, we obtain%
\begin{equation*}
\left\Vert \widehat{V}_{2s-1}^{+}\right\Vert _{\ell +\frac{1}{2},\Sigma
}\lesssim (\ell +1)!C^{\ell +1}(2s-1)^{2s-1}\gamma ^{2s-1}.
\end{equation*}%
\ Thus, (\ref{V2sminus_knorm}) becomes%
\begin{eqnarray*}
\frac{1}{k!}\left\vert V_{2s}^{-}\right\vert _{k,\Omega _{-}} &\lesssim
&C^{k+1}\underset{\ell =1}{\overset{k}{\dsum }}\frac{1}{\ell !}\left( (\ell
+1)!C^{\ell +1}(2s-1)^{2s-1}\gamma ^{2s-1}+C_{u^{-}}^{\ell +1}\left( \ell
+1\right) !(2s)!\gamma _{u^{-}}^{2s}+\left( 2s\right) !\widetilde{\gamma }%
_{f}^{2s}\right) + \\
&&+C^{k+1}(2s)!\gamma ^{2s} \\
&\lesssim &C_{1}^{k+1}(2s)^{2s}\gamma _{1}^{2s},
\end{eqnarray*}%
for suitable constants $C_{1},\gamma _{1}>0\ $independent of $\varepsilon $.
This establishes (\ref{Vjminus_bound}); to establish (\ref{Vjplus_bound})--(%
\ref{Vjplus_bound_derivatives}) we will simply check that the boundary data
in (\ref{Vjplus_even}) satisfies the appropriate bound (so that we may apply
Lemma \ref{lemmaMS}). Since 
\begin{equation*}
\widehat{V}_{2s}^{+}=V_{2s}^{-}-\left( u_{2s}^{+}-u_{2s}^{-}\right) \text{
on }\Sigma ,
\end{equation*}%
we see that for $\;\theta \in \lbrack 0,L_{\Sigma }],$ (cf. (\ref%
{Vjplus_even})) 
\begin{eqnarray*}
\left\vert \widehat{V}_{2s}^{+}\left( 0,\theta \right) \right\vert  &\leq
&\left\vert V_{2s}^{-}\left( 0,\theta \right) \right\vert +\left\vert
u_{2s}^{+}\left( 0,\theta \right) \right\vert +\left\vert u_{2s}^{-}\left(
0,\theta \right) \right\vert  \\
&\lesssim &\widehat{C}(2s)^{2s}\gamma ^{2s},
\end{eqnarray*}%
which is the bound that allows us to apply Lemma \ref{lemmaMS} and conclude that for $\widehat{V}%
_{2j}^{+}\left( \widehat{\rho },\theta \right) $, the estimates (\ref%
{Vjplus_bound}), (\ref{Vjplus_bound_derivatives}) hold as desired.

\noindent \textbf{The case of even $j$:} \ If $j$ is even, then $j+1$ is odd
and we would like to establish bounds for $V_{2s+1}^{-}$ and $\widehat{V}%
_{2s+1}^{+}$ (with $2s+1=j+1$), which satisfy (\ref{Vjminus_odd}), (\ref%
{Vjplus_odd}) respectively. First, for $V_{2s+1}^{-}$ we see from the
variational formulation of (\ref{Vjminus_odd}) that 
\begin{equation*}
\left\Vert V_{2s+1}^{-}\right\Vert _{1,\Omega _{-}}\lesssim \left\Vert 
\widehat{V}_{2s}^{+}\right\Vert _{0,\Sigma }\lesssim C_{+}(2s)^{2s}\gamma
^{2s}\leq C_{+}(2s+1)^{2s+1}\gamma ^{2s+1},
\end{equation*}
and, in a similar fashion as above, we obtain 
\begin{eqnarray*}
\frac{1}{k!}\left\vert V_{2s+1}^{-}\right\vert _{k,\Omega _{-}} &\leq
&C^{k+1}\left\{ \underset{\ell =0}{\overset{k-2}{\dsum }}\frac{1}{\ell !}%
\left\Vert \widehat{V}_{2s}^{+}\right\Vert _{\ell +1/2,\Sigma }+\left\Vert
V_{2s+1}^{-}\right\Vert _{1,\Omega _{-}}\right\} \\
&\lesssim &C^{k+1}\left\{ \underset{\ell =0}{\overset{k-2}{\dsum }}\frac{%
(\ell+1)!}{\ell !}\widehat{C}^{\ell+1}(2s)^{2s}\gamma
^{2s}+C_{+}(2s+1)^{2s+1}\gamma ^{2s+1}\right\} \\
&\lesssim &C_{2}^{k+1}(2s+1)^{2s+1}\gamma _{2}^{2s+1},
\end{eqnarray*}
for suitable constants $C_{2},\gamma _{2}>0$ independent of $\varepsilon $.
Finally, from the above result we see that the boundary data of (\ref%
{Vjplus_odd}) satisfies the appropriate bound, hence by Lemma \ref{lemmaMS}, $\widehat{V}_{2s+1}^{+}$
satisfies (\ref{Vjplus_bound}) and (\ref{Vjplus_bound_derivatives}) as
desired. \end{proof}

In view of the previous theorem, we define the (truncated) interface layer
expansion(s) as 
\begin{equation}
u_{IL}^{\varepsilon }:=\left( \widehat{v}_{I,M}^{+},v_{I,M}^{-}\right)
\label{ueIL}
\end{equation}%
where 
\begin{equation}
\widehat{v}_{I,M}^{+}=\underset{j=0}{\overset{2M+1}{\dsum }}\varepsilon ^{j}%
\widehat{V}_{j}^{+}\;,\;v_{I,M}^{-}=\underset{j=0}{\overset{2M+1}{\dsum }}%
\varepsilon ^{j}V_{j}^{-}.  \label{vIM}
\end{equation}%
The following corollary follows from Theorem \ref{thm_Vij}.

\begin{corollary}
\label{thm_vIM}
There exist constants $C,\gamma ,\Theta ,K>0$ depending only on the data, such that under the assumption
$\varepsilon (2M+1) \max \{\gamma ,K\}<1,$ the functions 
$\widehat{v}_{I,M}^{+},v_{I,M}^{-}$ defined by (\ref{vIM}) satisfy
\begin{equation*}
\left\vert v_{I,M}^{-}\right\vert _{k,\Omega _{-}}\lesssim C^{k+1}k!,
\end{equation*}
\begin{equation*}
\left\vert \partial _{\rho }^{p}\partial _{\theta }^{q}\widehat{v}%
_{I,M}^{+}(\rho ,\theta )\right\vert \lesssim \varepsilon
^{-p}e^{p}(p+1)^{1/2}q!\left( \frac{2}{\Theta }\right) ^{q},
\end{equation*}
for $p,q\in \mathbb{N}_{0},\rho \in \lbrack 0,\rho _{\Sigma }]$ and $\theta
\in S(\Theta )$ given by (\ref{Stheta}).
\end{corollary}%
\begin{proof}
By (\ref{vIM}) and Theorem \ref{thm_Vij} we have
\begin{equation*}
\left\vert v_{I,M}^{-}\right\vert _{k,\Omega _{-}}\leq \underset{j=0}{%
\overset{2M+1}{\dsum }}\varepsilon ^{j}\left\vert V_{j}^{-}\right\vert
_{k,\Omega _{-}}\lesssim \underset{j=0}{\overset{2M+1}{\dsum }}\varepsilon
^{j}k!C^{k+1}j^{j}\gamma ^{j}\lesssim k!C^{k+1}\underset{j=0}{\overset{2M+1}{%
\dsum }}\left( \varepsilon (2M+1)\gamma \right) ^{j}\lesssim k!C^{k+1}
\end{equation*}
and
\begin{eqnarray*}
\left\vert \partial _{\rho }^{p}\partial _{\theta }^{q}\widehat{v}%
_{I,M}^{+}(\rho ,\theta )\right\vert  &\leq &\underset{j=0}{\overset{2M+1}{%
\dsum }}\varepsilon ^{j}\left\vert \partial _{\rho }^{p}\partial _{\theta
}^{q}\widehat{V}_{j}^{+}(\rho ,\theta )\right\vert \lesssim \underset{j=0}{%
\overset{2M+1}{\dsum }}\varepsilon ^{j}\varepsilon
^{-p}e^{p}(p+1)^{1/2}q!\left( \frac{2}{\Theta }\right) ^{q}K^{j}j^{j} \\
&\lesssim &\varepsilon ^{-p}e^{p}(p+1)^{1/2}q!\left( \frac{2}{\Theta }%
\right) ^{q}\underset{j=0}{\overset{2M+1}{\dsum }}\left( \varepsilon
K(2M+1)\right) ^{j} \\
&\lesssim &\varepsilon ^{-p}e^{p}(p+1)^{1/2}q!\left( \frac{2}{\Theta }%
\right) ^{q}.
\end{eqnarray*}
\end{proof}

Finally in this section, we wish to see what the contribution of the
interface layers is, to the remainder of the expansion. For the interface
layers in $\Omega _{-}$ we easily see that 
\begin{equation}
-\Delta \left( v_{I,M}^{-}\right) +\left( v_{I,M}^{-}\right) =0.
\label{L1vIMminus}
\end{equation}%
Now, by construction of the functions $\widehat{V}_{2j}^{+}$ we have (with
the aid of (\ref{operator}) and (\ref{sumLepsilon})) 
\begin{eqnarray*}
L_{\varepsilon }\left( \widehat{v}_{I,M}^{+}\right) &=&\underset{i=2M+2}{%
\overset{\infty }{\dsum }}\varepsilon ^{i}\underset{j=0}{\overset{2M+1}{%
\dsum }}L_{i-j}\widehat{V}_{j}^{+} \\
&=&-\underset{j=0}{\overset{2M+1}{\dsum }}\underset{i=2M+2}{\overset{\infty }%
{\dsum }}\varepsilon ^{i}\widehat{\rho }^{i-1-j}a_{1}^{j-1-j}\partial _{%
\widehat{\rho }}\widehat{V}_{j}^{+}-\underset{j=0}{\overset{2M+1}{\dsum }}%
\underset{i=2M+3}{\overset{\infty }{\dsum }}\varepsilon ^{i}\widehat{\rho }%
^{i-2-j}a_{2}^{j-2-j}\partial _{\theta }^{2}\widehat{V}_{j}^{+}- \\
&&-\underset{j=0}{\overset{2M+1}{\dsum }}\underset{i=2M+4}{\overset{\infty }{%
\dsum }}\varepsilon ^{i}\widehat{\rho }^{i-2-j}a_{3}^{j-3-j}\partial
_{\theta }\widehat{V}_{j}^{+}.
\end{eqnarray*}%
By Lemma 2.12 of \cite{ms1}, we have the bound%
\begin{equation}
\left\vert L_{\varepsilon }\widehat{v}_{I,M}^{+}\left( \rho ,\theta \right)
\right\vert \lesssim K^{2M+2}\left( \varepsilon (2M+2)+\left\vert \rho
\right\vert \right) ^{2M+2}e^{-\rho /\varepsilon }\;\forall \;(\rho ,\theta
)\in B_{\rho _{0}}(0)\times S(\Theta ),  \label{L_e_vIPlus}
\end{equation}%
for some $K,\Theta >0$ independent of $\varepsilon $. (As before, $B_{\delta
}(z)$ denotes the open disc in the complex plane of \ radius $\delta $
centered at $z$, and $S(\Theta )$ is given by (\ref{Stheta}).)

\begin{remark}
Corollary \ref{thm_vIM} shows that the interface layer functions in $\Omega _{+}$
behave just like the boundary layers, while the interface layers in $\Omega
_{-}$ are smooth. This will be taken into consideration in the design of
the approximation scheme in Section \ref{approx} ahead.
\end{remark}

\subsection{Remainder estimates}

\label{rem} We now consider the remainder $r^{\varepsilon }\equiv \left(
r_{+}^{\varepsilon },r_{-}^{\varepsilon }\right) $ in the decomposition (\ref%
{decomp}), which is given by 
\begin{equation}
r_{+}^{\varepsilon }=u_{+}^{\varepsilon }-w_{+}^{\varepsilon }-\chi
_{BL}u_{BL}^{\varepsilon }-\chi _{IL}\widehat{v}_{I,M}^{+},
\label{r_epsilon_plus}
\end{equation}
\begin{equation}
r_{-}^{\varepsilon }=u_{-}^{\varepsilon }-w_{-}^{\varepsilon }-\chi
_{IL}v_{I,M}^{-},  \label{r_epsilon_minus}
\end{equation}
and by construction, satisfies the equivalent (but homogeneous) boundary
conditions as $u^{\varepsilon }$ on $\partial \Omega $. To see this note
that on $\partial \Omega _{-}\backslash \Sigma $ we have 
\begin{equation*}
\left. \left( r_{-}^{\varepsilon }\right) \right\vert _{\partial \Omega
_{-}\backslash \Sigma }=\left. \left( u_{-}^{\varepsilon
}-w_{-}^{\varepsilon }-\chi _{IL}v_{I,M}^{-}\right) \right\vert _{\partial
\Omega _{-}\backslash \Sigma }=0,
\end{equation*}
by (\ref{bvpd}), (\ref{u0-b}) and (\ref{uj-b}). \ On $\partial \Omega
_{+}\backslash \Sigma $ we have 
\begin{equation*}
\left. \left( r_{+}^{\varepsilon }\right) \right\vert _{\partial \Omega
_{+}\backslash \Sigma }=\left. \left( u_{+}^{\varepsilon
}-w_{+}^{\varepsilon }-\chi _{BL}u_{BL}^{\varepsilon }-\chi _{IL}\widehat{v}%
_{I,M}^{+}\right) \right\vert _{\partial \Omega _{+}\backslash \Sigma }=0,
\end{equation*}
by (\ref{bvpc}), (\ref{uBL_BC}) and (\ref{chi_IL}). Finally on $\Sigma $ we
have 
\begin{equation*}
\left. \left( r_{+}^{\varepsilon }-r_{-}^{\varepsilon }\right) \right\vert
_{\Sigma }=\left. \left( u_{+}^{\varepsilon }-u_{-}^{\varepsilon
}-w_{-}^{\varepsilon }+w_{+}^{\varepsilon }+\chi _{IL}v_{I,M}^{-}-\chi _{IL}%
\widehat{v}_{I,M}^{+}\right) \right\vert _{\Sigma }=0
\end{equation*}
by (\ref{bvpe}), (\ref{eq29}) and 
\begin{equation*}
\left. \left( \varepsilon ^{2}\frac{\partial r_{+}^{\varepsilon }}{\partial
\nu }-\frac{\partial r_{-}^{\varepsilon }}{\partial \nu }\right) \right\vert
_{\Sigma }=\left. \left( \varepsilon ^{2}\frac{\partial u_{+}^{\varepsilon }%
}{\partial \nu }-\frac{\partial u_{-}^{\varepsilon }}{\partial \nu }-\frac{%
\partial w_{-}^{\varepsilon }}{\partial \nu }+\varepsilon ^{2}\frac{\partial
w_{+}^{\varepsilon }}{\partial \nu }+\frac{\partial v_{I,M}^{-}}{\partial
\nu }-\varepsilon ^{2}\frac{\partial \widehat{v}_{I,M}^{+}}{\partial \nu }%
\right) \right\vert _{\Sigma }=0,
\end{equation*}
by (\ref{bvpf}), (\ref{u0-c}), (\ref{uj-c}) and (\ref{eq29}). Moreover, we
have the following.

\begin{theorem}
\label{thm_rem}
Let $r^{\varepsilon }=\left( r_{+}^{\varepsilon },r_{-}^{\varepsilon
}\right) $ be given by (\ref{r_epsilon_plus})--(\ref{r_epsilon_minus}) and let 
$L_{\varepsilon}r_{+}^{\varepsilon }=-\varepsilon ^{2}\Delta r_{+}^{\varepsilon
}+r_{+}^{\varepsilon }$ and $L_{1}r_{-}^{\varepsilon }=-\Delta
r_{-}^{\varepsilon }+r_{-}^{\varepsilon }$. \ Then there exist constants 
$K_{1},K_{2}>0$ independent of $\varepsilon $, such that
\begin{equation}
\left\Vert L_{\varepsilon }r_{+}^{\varepsilon }\right\Vert _{0,\Omega_{+}}
\lesssim \left( \varepsilon (2M+2)K_{1}\right) ^{2M+2}
\label{thm_rem1}
\end{equation}
and 
\begin{equation}
\left\Vert L_{1}r_{-}^{\varepsilon }\right\Vert _{0,\Omega_{-}}\lesssim
\left( \varepsilon (2M+2) K_{2} \right) ^{2M+2}.
\label{thm_rem2}
\end{equation}
\end{theorem}

\begin{proof}
We first consider (\ref{thm_rem1}) and we have
\begin{eqnarray}
L_{\varepsilon }r_{+}^{\varepsilon } &=&L_{\varepsilon }\left(
u_{+}^{\varepsilon }-w_{+}^{\varepsilon }-\chi _{BL}u_{BL}^{\varepsilon
}-\chi _{IL}\widehat{v}_{I,M}^{+}\right)   \notag \\
&=&L_{\varepsilon }\left( u_{+}^{\varepsilon }-w_{+}^{\varepsilon }\right)
-L_{\varepsilon }\left( \chi _{BL}u_{BL}^{\varepsilon }\right)
-L_{\varepsilon }\left( \chi _{IL}\widehat{v}_{I,M}^{+}\right) .
\label{Le_r+}
\end{eqnarray}
From (\ref{eq11}) we notice that
\begin{equation}
L_{\varepsilon }\left( u_{+}^{\varepsilon }-w_{+}^{\varepsilon }\right)
=\varepsilon ^{2M+2}\Delta ^{(M+1)}f_{+},  \label{Le_r_smooth}
\end{equation}
and also
\begin{equation*}
L_{\varepsilon }\left( \chi _{BL}u_{BL}^{\varepsilon }\right) =\varepsilon
^{2}\left( \Delta \chi _{BL}\right) u_{BL}^{\varepsilon }-2\varepsilon
^{2}\nabla \chi _{BL}\cdot \nabla u_{BL}^{\varepsilon }+\chi
_{BL}L_{\varepsilon }u_{BL}^{\varepsilon },
\end{equation*}
where the function $\chi _{BL}$ equals 1 for $0<\rho <\rho _{0}$ and 0 for 
$\rho >(\rho _{1}+\rho _{0})/2$. Hence by (\ref{uBL_bound}),
\begin{equation*}
\left\Vert \varepsilon ^{2}\left( \Delta \chi _{BL}\right)
u_{BL}^{\varepsilon }\right\Vert _{0,\Omega_{+}}\lesssim
\varepsilon ^{2}\left( 1+\left( \varepsilon (2M+1)K\right) ^{2M+1}\right)
e^{-\alpha \rho /\varepsilon }
\end{equation*}
and
\begin{equation*}
\left\Vert \varepsilon ^{2}\nabla \chi _{BL}\cdot \nabla u_{BL}^{\varepsilon
}\right\Vert _{0,\Omega_{+}}\lesssim \varepsilon \left(
1+\left( \varepsilon (2M+1)K\right) ^{2M+1}\right) e^{-\alpha \rho
/\varepsilon },
\end{equation*}
for some appropriate constant $K>0$. Therefore, by (\ref{LuBL_bound}) and
the previous two inequalities, we obtain
\begin{equation}
\left\Vert L_{\varepsilon }\left( \chi _{BL}u_{BL}^{\varepsilon }\right)
\right\Vert _{0,\Omega_{+}}\lesssim \left( \varepsilon
(2M+2)K\right) ^{2M+2}.  \label{L_chi_uBL}
\end{equation}
In a completely analogous way, we may obtain bounds for 
$L_{\varepsilon}\left( \chi _{IL}\widehat{v}_{I,M}^{+}\right) $, viz.
\begin{equation}
\left\Vert L_{\varepsilon }\left( \chi _{IL}\widehat{v}_{I,M}^{+}\right)
\right\Vert _{0,\Omega_{+}}\lesssim \left( \varepsilon
(2M+2)\widehat{K}\right) ^{2M+2},  \label{L_chi_vIMplus}
\end{equation}
for some appropriate constant $\widehat{K}>0$. 
Combining (\ref{Le_r+})--(\ref{L_chi_vIMplus}) we have
\begin{eqnarray*}
\left\Vert L_{\varepsilon }r_{+}^{\varepsilon }\right\Vert _{0,
\Omega_{+}} &\lesssim &\left\Vert \varepsilon ^{2M+2}\Delta
^{(M+1)}f_{+}\right\Vert _{0,\Omega_{+}}+\left(
\varepsilon (2M+2)K\right) ^{2M+2}+\left( \varepsilon (2M+2)\widehat{K}%
\right) ^{2M+2} \\
&\lesssim &\varepsilon ^{2M+2}\gamma _{f_{+}}^{2M+2}(2M+2)!+\left(
\varepsilon (2M+2)K\right) ^{2M+2}+\left( \varepsilon (2M+2)\widehat{K}%
\right) ^{2M+2} \\
&\lesssim &\left( \varepsilon (2M+2)K_{1}\right) ^{2M+2},
\end{eqnarray*}
for a suitable $K_{1}>0$ independent of $\varepsilon $. This establishes (\ref{thm_rem1}).

Turning our attention to (\ref{thm_rem2}), we have
\begin{eqnarray}
L_{1}r_{-}^{\varepsilon} &=&L_{1}\left( u_{-}^{\varepsilon
}-w_{-}^{\varepsilon }-\chi _{IL}v_{I,M}^{-}\right)   \notag \\
&=&L_{1}\left( u_{-}^{\varepsilon }-w_{-}^{\varepsilon }\right) -L_{1}\left(
\chi _{IL}v_{I,M}^{-}\right) .  \label{L1_r-}
\end{eqnarray}
We have from (\ref{L1_r-}) (with the aid of (\ref{eq18}))
\begin{equation*}
L_{1}r_{-}^{\varepsilon }=-L_{1}\left( \chi _{IL}v_{I,M}^{-}\right) ,
\end{equation*}
hence by (\ref{L1vIMminus}),
\begin{eqnarray*}
\left\Vert L_{1}r_{-}^{\varepsilon }\right\Vert _{0,\Omega _{-}}
&=&\left\Vert L_{1}\left( \chi _{IL}v_{I,M}^{-}\right) \right\Vert
_{0\Omega _{-}}\\
&=&\left\Vert \left( \Delta \chi _{IL}\right)
v_{I,M}^{-}-2\nabla \chi _{IL}\cdot \nabla v_{I,M}^{-}+\chi
_{IL}L_{1}v_{I,M}^{-}\right\Vert _{0,\Omega _{-}} \\
&\leq &\left\Vert \left( \Delta \chi _{IL}\right) v_{I,M}^{-}\right\Vert
_{0,\Omega _{-}}+\left\Vert 2\nabla \chi _{IL}\cdot \nabla
v_{I,M}^{-}\right\Vert _{0,\Omega _{-}}.
\end{eqnarray*}
Since the function $\chi _{IL}$ equals 1 for $0<\rho <\rho _{\Sigma }$ and 0
for $\rho >(\rho _{2}+\rho _{0})/2$ (cf. (\ref{chi_IL})), we further get
(using (\ref{Vjminus_bound}))
\begin{eqnarray*}
\left\Vert L_{1}r_{-}^{\varepsilon }\right\Vert _{0,\Omega _{-}} &\lesssim &%
\underset{j=0}{\overset{2M+1}{\dsum }}\varepsilon ^{j}\left( \left\Vert
V_{j}^{-}\right\Vert _{0,\Omega _{-}} + \left\Vert
\nabla V_{j}^{-}\right\Vert _{0,\Omega _{-}} \right) \lesssim \underset{j=0}{\overset{2M+1}%
{\dsum }}\varepsilon ^{j} j^{j}\gamma ^{j} \\
&\lesssim &\underset{j=0}{\overset{2M+1}{\dsum }}\left(
\varepsilon \gamma (2M+1)\right) ^{j} \lesssim \left( \varepsilon K_{2}(2M+2)\right) ^{2M+2}
\end{eqnarray*}
for a suitable $K_{2}>0$ independent of $\varepsilon $. Thus (\ref{thm_rem2}) 
is established and this completes the proof.
\end{proof}

\begin{remark}
Theorem  \ref{thm_rem}  shows that for $\varepsilon M$ sufficiently small,  the remainder in 
(\ref{decomp}) is exponentially small, hence it need not be approximated. This information will be 
utilized in the next section when we will construct the approximation to $u^{\varepsilon}$.
\end{remark}

\section{Approximation results\label{approx}}


We begin this section with the variational formulation of (\ref{bvpa})--(\ref%
{bvpf}), which reads: Find $u^{\varepsilon }=\left( u_{+}^{\varepsilon
},u_{-}^{\varepsilon }\right) \in H_{0}^{1}\left( \Omega \right) $ such that 
\begin{equation}
B_{\varepsilon }\left( u^{\varepsilon },v\right) =F(v)\;\forall \;v=\left(
v_{+}^{\varepsilon },v_{-}^{\varepsilon }\right) \in H_{0}^{1}\left( \Omega
\right) ,  \label{VF}
\end{equation}%
where 
\begin{equation}
B_{\varepsilon }\left( u^{\varepsilon },v\right) =\dint\nolimits_{\Omega
_{+}}\left\{ \varepsilon ^{2}\nabla \left( u_{+}^{\varepsilon }\right) \cdot
\nabla \left( v_{+}^{\varepsilon }\right) +u_{+}^{\varepsilon
}v_{+}^{\varepsilon }\right\} +\dint\nolimits_{\Omega _{-}}\left\{ \nabla
\left( u_{-}^{\varepsilon }\right) \cdot \nabla \left( v_{-}^{\varepsilon
}\right) +u_{-}^{\varepsilon }v_{-}^{\varepsilon }\right\} ,  \label{Buv}
\end{equation}%
\begin{equation}
F(v)=\dint\nolimits_{\Omega _{+}}f_{+}v_{+}^{\varepsilon
}+\dint\nolimits_{\Omega _{-}}f_{-}v_{-}^{\varepsilon
}+\dint\nolimits_{\Sigma }hv.  \label{Fv}
\end{equation}%
It is straight forward to show that the bilinear form (\ref{Buv}) is
coercive and continuous on $H_{0}^{1}\left( \Omega \right) $, hence the
variational problem (\ref{VF}) admits a unique solution thanks to the
Lax-Milgram lemma. The discrete version of (\ref{VF}) reads: Find $%
u_{N}^{\varepsilon }=\left( u_{+}^{N},u_{-}^{N}\right) \in V_{N}\subset
H_{0}^{1}\left( \Omega \right) $ such that 
\begin{equation}
B_{\varepsilon }\left( u_{N}^{\varepsilon },v\right) =F(v)\;\forall
\;v=\left( v_{+}^{\varepsilon },v_{-}^{\varepsilon }\right) \in V_{N}\subset
H_{0}^{1}\left( \Omega \right) ,  \label{Buv_N}
\end{equation}%
and by Céa's Lemma we have 
\begin{equation}
\left\Vert u^{\varepsilon }-u_{N}^{\varepsilon }\right\Vert _{\varepsilon
}\leq \underset{v\in V_{N}}{\inf }\left\Vert u^{\varepsilon }-v\right\Vert
_{\varepsilon },  \label{best_approximation}
\end{equation}%
where the energy norm $\left\Vert \cdot \right\Vert _{\varepsilon }$ is
defined as 
\begin{equation}
\left\Vert u\right\Vert _{\varepsilon }^{2}=B_{\varepsilon }\left(
u,u\right) .  \label{energy_norm}
\end{equation}%
We now describe the subspace $V_{N}$. For simplicity, we will focus on
quadrilateral elements, even though triangular elements are also possible
(see \cite{ms} for this and other choices of a suitable mesh). Since the
behavior of the solution $u^{\varepsilon }$ depends on the value of $%
\varepsilon $ (cf. Theorem \ref{thm_asy}), we distinguish between the cases $%
\kappa p\varepsilon \geq 1/2$ and $\kappa p\varepsilon <1/2$ (with $\kappa
\in \mathbb{R}$ \ a fixed constant) as follows: If $\kappa p\varepsilon \geq
1/2$ then the mesh does not need any special design, as in this case the
polynomial degree $p$ of the approximating functions is high enough to
ensure good approximability. Hence, in this case the mesh $\Delta $ only
needs to be regular in the sence of \cite{ciarlet:78} (or satisfy conditions
M1--M3 in \cite{ms}). \ In the case $\kappa p\varepsilon <1/2$ the mesh will
include elements of size $O(p\varepsilon )$ along $\partial \Omega _{+}$ in
order for the boundary and interface layer effects to be captured -- these
are referred to as \emph{needle elements} in \cite{ms}. \ We now describe
one such possible construction: Let $\Omega _{+}^{0}$ be given by (\ref%
{Omega_0}), and divide $\partial \Omega _{+}\backslash \Sigma $ into
subintervals $\left( \theta _{j},\theta _{j+1}\right) ,j=1,...,m-1,\theta
\in \partial \Omega _{+}$. Then draw the inward normal at $\theta _{j}$ of
length $\rho _{0}$ (see eq. (\ref{rho_0})) and connect each point $\left(
\rho _{j},\theta _{j}\right) =\left( \rho _{0},\theta _{j}\right) $ using
the curve $\rho =\rho _{0}$ (=constant). \ Further, divide each 
\begin{equation}
\left( \Omega _{+}^{0}\right) _{j}:=\left\{ \left( \rho ,\theta \right)
:0\leq \rho \leq \rho _{0},\theta _{j}\leq \theta \leq \theta _{j+1}\right\}
,j=1,...,m  \label{Omega_0_j}
\end{equation}%
into $\left( \Omega _{+}^{0,1}\right) _{j},\left( \Omega _{+}^{0,2}\right)
_{j}$ , where 
\begin{eqnarray*}
\left( \Omega _{+}^{0,1}\right) _{j} &=&\left\{ \left( \rho ,\theta \right)
:\theta _{j}\leq \theta \leq \theta _{j+1},0\leq \rho \leq \frac{1}{2}\rho
_{0}\kappa p\varepsilon \right\} , \\
\left( \Omega _{+}^{0,2}\right) _{j} &=&\left( \Omega _{+}^{0}\right)
_{j}\backslash \left( \Omega _{+}^{0,1}\right) _{j}.
\end{eqnarray*}%
In the above definitions, $\kappa \in \mathbb{R}$ is a fixed constant, $p$
is the degree of the approximating polynomials and we recall that we assume $%
\kappa p\varepsilon <1/2.$ This will define a mesh 
\begin{equation*}
\Delta _{+}^{0}:=\left\{ \left( \Omega _{+}^{0,1}\right) _{j},\left( \Omega
_{+}^{0,2}\right) _{j}\right\} _{j=1}^{m}
\end{equation*}%
over $\Omega _{+}^{0}.$ We may define a completely analogous mesh $\Delta
_{\Sigma }^{0}$ over $\Omega _{\Sigma }^{0}$ (see eq. (\ref{Omega_sigma})),
as 
\begin{equation*}
\Delta _{\Sigma }^{0}:=\left\{ \left( \Omega _{\Sigma }^{0,1}\right)
_{k},\left( \Omega _{\Sigma }^{0,2}\right) _{k}\right\} _{k=1}^{n},
\end{equation*}%
with $\left( \Omega _{\Sigma }^{0,1}\right) _{k},\left( \Omega _{\Sigma
}^{0,2}\right) _{k}$ defined in an analogous way as $\left( \Omega
_{+}^{0,1}\right) _{j},\left( \Omega _{+}^{0,2}\right) _{j}$. Next, let $%
\left\{ \left( \Omega _{+}^{1}\right) _{i}\right\} _{i=1}^{\ell }$ be some
subdivision of $\Omega _{+}^{1}$ that is compatible with $\Delta _{+}^{0}$
and $\Delta _{\Sigma }^{0},$ and define the mesh 
\begin{equation}
\Delta _{+}=\left\{ \left( \Omega _{\Sigma }^{0,1}\right) _{k},\left( \Omega
_{\Sigma }^{0,2}\right) _{k},\left( \Omega _{+}^{0,1}\right) _{j},\left(
\Omega _{+}^{0,2}\right) _{j},\left( \Omega _{+}^{1}\right)
_{i},k=1,...,n,j=1,...,m,i=1,...,\ell \right\} ,  \label{Delta_+}
\end{equation}%
over $\Omega _{+}.$ \ The mesh $\Delta _{-}$ over $\Omega _{-}$ is simply be
chosen to be compatible with $\Delta _{+}$ , and regular, in the sense of 
\cite{ciarlet:78}. The mesh over the entire domain $\Omega $ is then taken
to be 
\begin{equation}
\Delta =\Delta _{+}\cup \Delta _{-},  \label{Delta}
\end{equation}%
and we assume that the number of elements in $\Delta $ is bounded
independently of $\varepsilon $. The above mesh satisfies the definition of
a \emph{regular admissible boundary layer mesh} (Definition 3.2 in \cite{ms}%
), which implies the following: With $S:=[0,1]\times \lbrack 0,1]$ the usual
reference square, we associate with each quadrilateral $\Omega _{i}^{\pm
}\in \Delta $ a differentiable, bijective element mapping 
\begin{equation*}
M_{i}^{\pm }:S\rightarrow \overline{\Omega }_{i}^{\pm },
\end{equation*}%
which, in this case, satisfies $\forall \;i$ 
\begin{equation*}
\left\Vert D^{\alpha }M_{i}^{\pm }\right\Vert _{L^{\infty }(S)}\lesssim
\gamma ^{\left\vert \alpha \right\vert }\left\vert \alpha \right\vert
!\;\forall \;\alpha \in \mathbb{N}_{0}^{2}.
\end{equation*}%
The space $V_{N}$ is then defined as 
\begin{equation}
V_{N}=\left\{ u\in H^{1}\left( \Omega \right) :\left. u\right\vert _{\Omega
_{i}^{\pm }}=\phi _{p}\circ \left( M_{i}^{\pm }\right) ^{-1}\text{ for }\phi
_{p}\in Q_{p}(S)\right\} \cap H_{0}^{1}\left( \Omega \right) ,  \label{VN}
\end{equation}%
where $Q_{p}(S)$ denotes the space of all polynomials of degree $p$ in each
variable defined on the reference square $S$. Note that 
\begin{equation*}
N=\dim V_{N}=O\left( p^{2}\right) .
\end{equation*}

Now, for $p\geq 1$ we define on the space of continuous function $C\left(
[0,1]\right) ,$ the operator $\pi _{p}$ by interpolation in the $p+1$
Gauss-Lobatto points, and on $S$ we introduce the interpolation operator $%
\Pi _{p}$ as the tensor product of the two one-dimensional operators $\pi
_{p}^{x}$ and $\pi _{p}^{y}$. Then, by Lemma 3.8 in \cite{ms}, we have that
for any $u\in C^{\infty }\left( S\right) $ with $\left\Vert D^{\alpha
}u\right\Vert _{0,S}\lesssim \gamma ^{\left\vert \alpha \right\vert
}\left\vert \alpha \right\vert !\;\forall \;\alpha \in \mathbb{N}_{0}^{2},$
there exists a constant $\sigma >0$ depending only on $\gamma $, such that 
\begin{equation}
\left\Vert u-\Pi _{p}u\right\Vert _{L^{\infty }(S)}+\left\Vert \nabla \left(
u-\Pi _{p}u\right) \right\Vert _{L^{\infty }(S)}\lesssim e^{-\sigma p}.
\label{interpolation}
\end{equation}
Moreover, there holds (see, e.g., Lemma 3.7 in \cite{ms}), 
\begin{equation}
\left\{ 
\begin{array}{c}
\left\Vert \Pi _{p}u\right\Vert _{L^{\infty }(S)}\lesssim (1+\ln
p)^{2}\left\Vert u\right\Vert _{L^{\infty }(S)} \\ 
\left\Vert \partial _{x}\Pi _{p}u\right\Vert _{L^{\infty }(S)},\left\Vert
\partial _{y}\Pi _{p}u\right\Vert _{L^{\infty }(S)}\lesssim p^{2}(1+\ln
p)^{2}\left\Vert u\right\Vert _{L^{\infty }(S)}%
\end{array}%
\right. \;.  \label{interpolation2}
\end{equation}

We now prove our main approximation result.

\begin{theorem}
\label{thm_main}
Let $u^{\varepsilon }\in H_{0}^{1}\left( \Omega \right) ,u_{N}^{\varepsilon
}\in V_{N}$ be the solutions of (\ref{Buv}) and (\ref{Buv_N}), respectively, with
$V_N$ defined by (\ref{VN}) on the mesh $\Delta$ given by (\ref{Delta}).  Further,
assume that $\partial \Omega$ is analytic and the functions $f_{\pm }$ are 
analytic on $\Omega _{\pm }$ while the function $h$ is analytic on $\Sigma $. Then,
for $\kappa$ sufficiently small, we have
\begin{equation*}
\left\Vert u^{\varepsilon }-u_{N}^{\varepsilon }\right\Vert _{\varepsilon
}\lesssim N^2 e^{-b\sqrt{N}},
\end{equation*}
for some constant $b>0$ independent of $\varepsilon $ and $p$.
\end{theorem}%
\begin{proof} We consider the cases $\kappa p\varepsilon >1/2$ (asymptotic
case) and $\kappa p\varepsilon \leq 1/2$ (pre-asymptotic case) separately.

\emph{Case 1}: $\kappa p\varepsilon >1/2$ (asymptotic case)

By Theorem \ref{thm_asy} there exist constants $C,K>0$
depending only on the data such that
\begin{equation*}
\varepsilon\left\Vert D^{\alpha }u_{+}^{\varepsilon }\right\Vert _{0,\Omega _{+}}+
\left\Vert D^{\alpha }u_{-}^{\varepsilon }\right\Vert _{0,\Omega _{-}}\leq
C \varepsilon K^{\left\vert \alpha \right\vert }\max \left\{ \left\vert \alpha
\right\vert ,\varepsilon ^{-1}\right\} ^{\left\vert \alpha \right\vert
}\;\forall \;\alpha = 1,2,... .  
\end{equation*}
Now, by Lemma 3.10 of \cite{ms} we have that for each element map $M_{i}^{\pm }$,
\begin{equation*}
\left\Vert D^{\alpha }\left( u_{\pm}^{\varepsilon }\circ M_{i}^{\pm}\right)
\right\Vert _{0,S}\lesssim \gamma _{\pm}^{\left\vert \alpha \right\vert
}\left\vert \alpha \right\vert !e^{1/\varepsilon },
\end{equation*}
for some constants $\gamma _{\pm }>0$ independent of $\varepsilon $ and $i$.
Hence, by (\ref{interpolation}) and Lemma 3.6 of \cite{ms} we have 
\begin{equation*}
\left\Vert \left( u_{\pm }^{\varepsilon }\circ M_{i}^{\pm }\right) -\Pi _{p}\left(
u_{\pm }^{\varepsilon }\circ M_{i}^{\pm }\right) \right\Vert _{L^{\infty
}(S)}+\left\Vert \nabla \left( \left( u_{\pm}^{\varepsilon }\circ M_{i}^{\pm
}\right) -\Pi _{p}\left( u_{\pm }^{\varepsilon }\circ M_{i}^{\pm }\right) \right)
\right\Vert _{L^{\infty }(S)}\lesssim e^{-\sigma p+1/\varepsilon }.
\end{equation*}
Since $2\kappa p>1/\varepsilon ,$ we have $-\sigma p+1/\varepsilon \leq
-\sigma p+2\kappa p$ and, under the assumption that $\kappa <\sigma /2,$
\begin{equation*}
\left\Vert \left( u_{\pm }^{\varepsilon }\circ M_{i}^{\pm }\right) -\Pi _{p}\left(
u_{\pm }^{\varepsilon }\circ M_{i}^{\pm }\right) \right\Vert _{L^{\infty
}(S)}+\left\Vert \nabla \left( \left( u_{\pm }^{\varepsilon }\circ M_{i}^{\pm
}\right) -\Pi _{p}\left( u_{\pm }^{\varepsilon }\circ M_{i}^{\pm }\right) \right)
\right\Vert _{L^{\infty }(S)}\lesssim e^{-bp},
\end{equation*}
for some constant \ $b>0$. \ By (\ref{best_approximation}) and 
(\ref{energy_norm}) we get the desired result.

\emph{Case 2}: $\kappa p\varepsilon \leq 1/2$ (pre-asymptotic case)

In this case we utilize the expansion and regularity results of Section 
\ref{expansion} which state that the solution $u^{\varepsilon }=\left(
u_{+}^{\varepsilon },u_{-}^{\varepsilon }\right) $ can be written as
\begin{eqnarray*}
u_{+}^{\varepsilon } &=&w_{M}^{+}+\chi _{BL}u_{BL}^{M}+\chi _{IL}\widehat{v}%
_{I,M}^{+}+r_{+}^{\varepsilon }, \\
u_{-}^{\varepsilon } &=&w_{M}^{-}+\chi _{IL}v_{I,M}^{-}+r_{-}^{\varepsilon },
\end{eqnarray*}
with each term defined and analyzed in subsections \ref{smooth}--\ref{rem}.
We begin be selecting $M$ in such a way that $\varepsilon M$ is sufficiently
small for all the regularity results of subsections \ref{smooth}--\ref{rem}
to hold true. (The lack of concreteness on our part is due to the careful
constant selection made in \cite{ms}, i.e. such a choice for $M$ is possible
by \cite{ms}). The proof relies on the following observations:

\begin{enumerate}
\item The terms $w_{M}^{\pm },v_{I,M}^{-}$ are analytic in their respective
domains, hence (\ref{interpolation}) may be applied.

\item The construction of the mesh allows us to approximate $u_{BL}^{M}$ and 
$\widehat{v}_{I,M}^{+}$ at an exponential rate.

\item The choice of $M$ renders the term $\Vert r^{\varepsilon } \Vert_{\varepsilon, \Omega}$ 
negligible (exponentially small), hence the remainder need not be approximated.
\end{enumerate}

Let us first consider item 1 above. \ For the term $v_{I,M}^{-}$ we have, by
Corollary \ref{thm_vIM}, that for each element map $M_{i}^{-}$
\begin{equation*}
\left\Vert D^{\alpha }\left( v_{I,M}^{-}\circ M_{i}^{-}\right) \right\Vert
_{0,S}\lesssim C^{\left\vert \alpha \right\vert }\left\vert \alpha
\right\vert !,
\end{equation*}
hence by (\ref{interpolation}), 
\begin{equation*}
\left\Vert \left( v_{I,M}^{-}\circ M_{i}^{-}\right) -\Pi _{p}\left(
v_{I,M}^{-}\circ M_{i}^{-}\right) \right\Vert _{L^{\infty }(S)}+\left\Vert
\nabla \left( \left( v_{I,M}^{-}\circ M_{i}^{-}\right) -\Pi _{p}\left(
v_{I,M}^{-}\circ M_{i}^{-}\right) \right) \right\Vert _{L^{\infty
}(S)}\lesssim e^{-bp}.
\end{equation*}
The same works for the other two terms (the details are ommitted).

Next let us comment on item 2; since the steps are very similar for both 
$u_{BL}^{M}$ and $\widehat{v}_{I,M}^{+}$, we will only consider the latter.
Without loss of generality we assume that $\chi _{IL}$ is 1 in 
$\Omega_{\Sigma }^{0}$ and 0 otherwise. Hence we only need to approximate 
$\widehat{v}_{I,M}^{+}$ within $\Omega _{\Sigma }^{0}$. To this end, we note that the
mesh in $\Omega _{\Sigma }^{0}$ consists of two types of elements (cf. (\ref{Delta_+})): 
\begin{equation*}
\left( \Omega _{\Sigma }^{0,1}\right) _{j}=\left\{ \left( \rho ,\theta
\right) :\theta _{j}\leq \theta \leq \theta _{j+1},0\leq \rho \leq \frac{1}{2%
}\rho _{0}\kappa p\varepsilon \right\} \;\text{\ and }\;\left( \Omega
_{\Sigma }^{0,2}\right) _{j}=\left( \Omega _{+}^{0}\right) _{j}\backslash
\left( \Omega _{+}^{0,1}\right) _{j}.
\end{equation*}
For $\left( \Omega _{\Sigma }^{0,1}\right) _{j}$ , with associated mapping 
$M_{j}^{0,1}$, we have from Proposition 3.11 in \cite{ms}
\begin{equation*}
\left\Vert D^{\alpha }\left( \widehat{v}_{I,M}^{+}\circ \psi ^{-1}\circ
M_{j}^{0,1}\right) \right\Vert _{L^{\infty }(S)}\lesssim e^{\kappa
p}K^{\left\vert \alpha \right\vert }\left\vert \alpha \right\vert !\;\forall
\alpha \in \mathbb{N}_{0}^{2},
\end{equation*}
where $\psi $ was defined by (\ref{psi}). \ Therefore, by (\ref{interpolation})
\begin{equation*}
\left\Vert \left( \widehat{v}_{I,M}^{+}\circ \psi ^{-1}\circ
M_{j}^{0,1}\right) -\Pi _{p}\left( \widehat{v}_{I,M}^{+}\circ \psi
^{-1}\circ M_{j}^{0,1}\right) \right\Vert _{L^{\infty }(S)}+
\end{equation*}
\begin{equation*}
+\left\Vert \nabla \left( \left( \widehat{v}_{I,M}^{+}\circ \psi ^{-1}\circ
M_{j}^{0,1}\right) -\Pi _{p}\left( \widehat{v}_{I,M}^{+}\circ \psi
^{-1}\circ M_{j}^{0,1}\right) \right) \right\Vert _{L^{\infty }(S)}\lesssim
e^{\kappa p}e^{-bp},
\end{equation*}
from which the desired result follows \emph{provided} $\kappa <b.$ Now, let
us consider the approximation of $\widehat{v}_{I,M}^{+}$ over the elements 
$\left( \Omega _{\Sigma }^{0,2}\right) _{j}$, with associated mapping 
$M_{j}^{0,1}$. From (\ref{uBL_bound}) we have 
\begin{equation*}
\left\Vert \chi _{IL}\widehat{v}_{I,M}^{+}\circ \psi ^{-1}\circ
M_{j}^{0,1}\right\Vert _{L^{\infty }(S)}\lesssim C_{\alpha }e^{-\alpha
\kappa p}\;,\;\left\Vert \nabla \left( \chi _{IL}\widehat{v}_{I,M}^{+}\circ
\psi ^{-1}\circ M_{j}^{0,1}\right) \right\Vert _{L^{\infty }(S)}\lesssim
C_{\alpha }\varepsilon ^{-1}e^{-\alpha \kappa p}.
\end{equation*}
Therefore, from (\ref{interpolation2}) we get
\begin{equation*}
\left\Vert \chi _{IL}\widehat{v}_{I,M}^{+}\circ \psi ^{-1}\circ
M_{j}^{0,1}\right\Vert _{L^{\infty }(S)}\lesssim \left( 1+\ln p\right)
^{2}e^{-bp}
\end{equation*}
and
\begin{equation*}
\left\Vert \nabla \left( \chi _{IL}\widehat{v}_{I,M}^{+}\circ \psi
^{-1}\circ M_{j}^{0,1}\right) \right\Vert _{L^{\infty }(S)}\lesssim
\varepsilon ^{-1}p^{2}\left( 1+\ln p\right) ^{2}e^{-bp},
\end{equation*}
from which the desired result follows once we use (\ref{energy_norm}).

Turning to item 3, we have from the variational formulation (\ref{VF})--(\ref{Fv}), that the remainder
$r^{\varepsilon} = \left( r^{\varepsilon}_{+}, r^{\varepsilon}_{-} \right)$ satisfies
\[
B_{\varepsilon }(r^{\varepsilon }, v)=
\int_{ \Omega _{+}}  L_{\varepsilon }r_{+}^{\varepsilon } v\,dx+
\int_{ \Omega _{-}}  L_{1 }r_{-}^{\varepsilon } v\,dx, \quad \forall v\in H^1_0(\Omega).
\]
Hence 
by Cauchy-Schwarz's inequality we get
\[
B_{\varepsilon }(r^{\varepsilon }, r^{\varepsilon })\leq
\|L_{\varepsilon }r_{+}^{\varepsilon }\|^2_{0, \Omega _{+}}+
\|L_{1 }r_{-}^{\varepsilon }\|^2_{0, \Omega _{-}},
\]
and by Theorem \ref{thm_rem} we obtain
\[
\Vert r^{\varepsilon} \Vert_{\varepsilon} = %
\left[ B_{\varepsilon }(r^{\varepsilon }, r^{\varepsilon })\right]^{1/2}
\lesssim \left( \varepsilon (2M+2)K\right) ^{2M+2} \lesssim e^{-bp},
\]
%
%
for some suitable constant $b>0$, depending only on the data. This completes the proof. \end{proof}

\section{Numerical results\label{computations}}

In this section we will illustrate our theoretical findings for the model
problem (\ref{bvpa})--(\ref{bvpf}), in the case when $f_{+}=f_{-}=1,h=0$ and
the domain $\Omega $ consists of the two subdomains $\Omega _{+}$ and $%
\Omega _{-}$, delimited by the three concentric circles with radii 1, 2 and
3. In other words, $\Omega _{+}$ is the domain inside the two concentric
circles of radii 1 and 2, while $\Omega _{-}$ is the domain inside the two
concentric circles of radii 2 and 3, as shown in figure 2.

\begin{figure}[h]
\label{compdomain} \setlength{\unitlength}{1cm}
\par
\begin{center}
\hspace{2cm}\psfig{figure=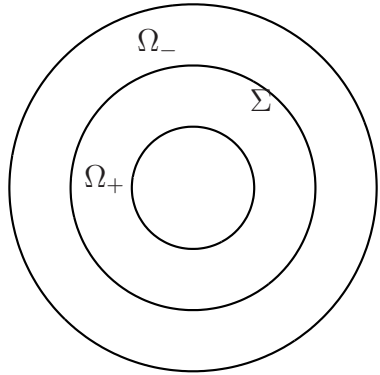,width=5cm} 
\begin{picture}(2,3)
\put (-4.1,2.5) {\shortstack[c] {$\Omega_{+}$}}
\put (-3.4,4.3) {\shortstack[c] {$\Omega_{-}$}}
\put (-1.9,3.5) {\shortstack[c] {$\Sigma$}}
\end{picture}
\end{center}
\caption{Domains $\Omega_+$ and $\Omega_-$ used for the computations.}
\end{figure}

We expect to have a boundary layer along $\partial \Omega _{+}\backslash
\Sigma $ (the circle of radius 1) and an interface layer along $\Sigma $
(the circle of radius 2). The mesh, shown in figure 
3, accounts for the presence of the layers by including thin elements of
size $p\varepsilon $ along $\partial \Omega _{+}\backslash \Sigma $ and $%
\Sigma $ -- the value of the constant $\kappa $ appearing in the definition
of the mesh in the previous section was taken to be 1 (a value known to
produce almost the same results as those obtained with the \textquotedblleft
optimal \textquotedblright\ value of $\kappa $, see, e.g., \cite{ss}). An
exact solution is available for this problem, hence our computations are
reliable.

\begin{figure}[h]
\label{mesh}
\par
\begin{center}
\psfig{figure=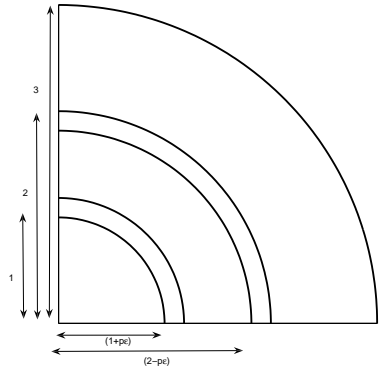,width=5cm}
\end{center}
\caption{Design of the mesh (on a quarter of the domain).}
\end{figure}

\begin{figure}[h]
\label{soln}
\par
\begin{center}
\psfig{figure=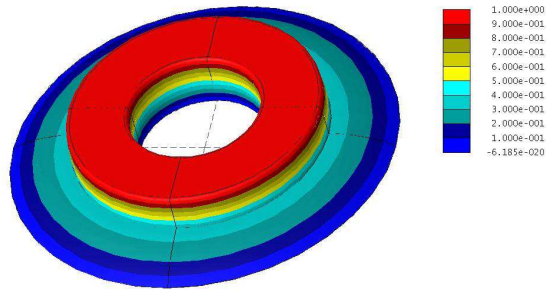,width=3in} 
\end{center}
\caption{Approximate solution for $p=8, \protect\varepsilon =0.01$.}
\end{figure}

\begin{figure}[h]
\label{conv}
\par
\begin{center}
\psfig{figure=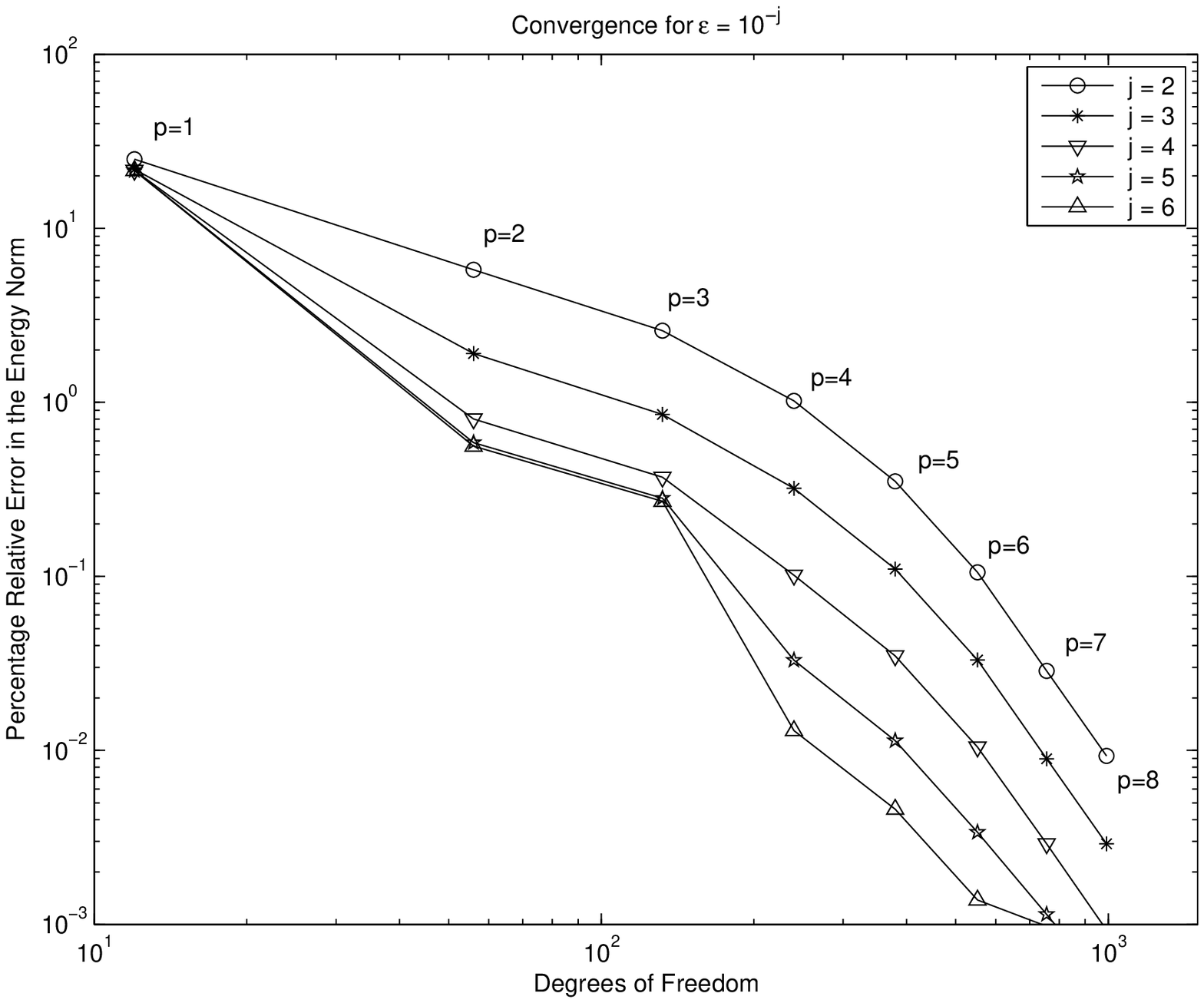,width=3in} %
\psfig{figure=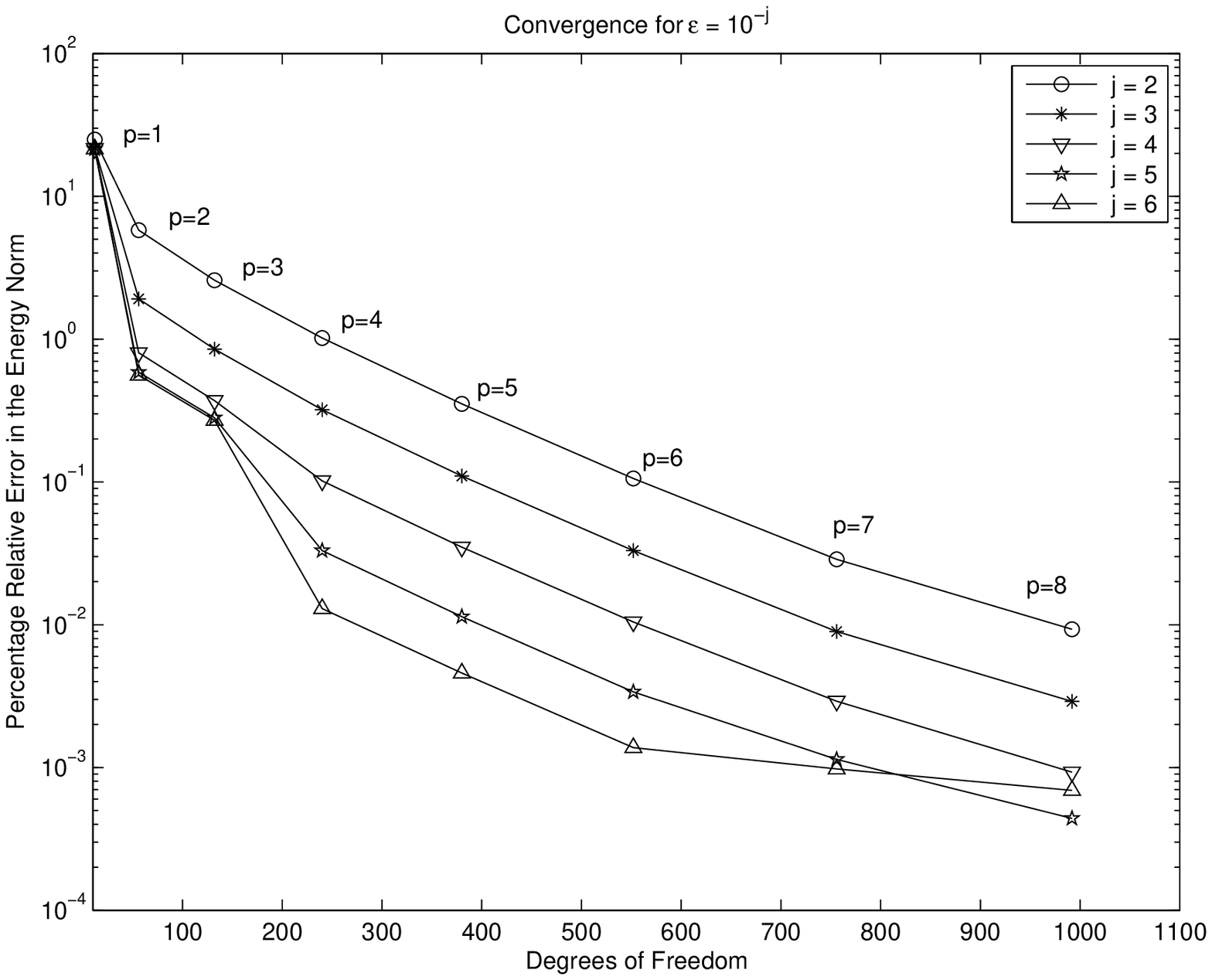,width=3in} 
\end{center}
\caption{Convergence of the approximate solution: loglog plot(left); semilog
plot (right).}
\end{figure}

The computations were performed with the commercial package StressCheck
(E.S.R.D., St. Louis, MO) which is a $p$-version FEM software package
allowing the polynomial degree to vary from $p=1$ to $p=8$ (on a fixed
mesh). Figure 4 below shows the approximate solution for $p=8, \varepsilon =
0.01$, and figure 5 shows the convergence (in the energy norm) as $p$ is
increased -- the exponential convergence is readily visible.

\section{Conclusions}

\label{conclusions} We have studied the finite element approximation of a
singularly perturbed transmission problem posed on a (smooth) domain with
analytic boundary. Upon obtaining appropriate regularity results, via
asymptotic expansions, we were able to design and analyze an $hp$ finite
element method for the robust approximation of the solution to the
singularly perturbed transmission problem. \ We showed that under the
assumption of analytic data, our method converges at an exponential rate,
independently of the singular perturbation parameter. This is in line with
our one-dimensional results \cite{nx}, as well as with two-dimensional
results for non-transmission problems \cite{ms}. The approximation of
singularly perturbed transmission problems on non-smooth domains is the
focus of our current research efforts.

\end{document}